\documentclass[11pt]{article}
\usepackage{euscript,yfonts,dsfont,textcomp,amsmath,amssymb,amstext,parskip,url,graphicx,tikz,float,amsthm}
\usepackage{enumitem}
\usepackage{picins}

\usepackage{algorithm}
\usepackage{algorithmic}
\usepackage{amsfonts}
\usepackage{booktabs} 
\usepackage{xspace}
\usepackage{multirow}
\usepackage{tabularx,bbm}
\usepackage{xcolor}

\usepackage[colorlinks=true,
            linkcolor=blue,
            citecolor=blue,
            urlcolor=blue]{hyperref}
        
\usepackage[margin=1in]{geometry}

\newcommand{\calV}{{\mathcal V}}
\newcommand{\calE}{{\mathcal E}}
\newcommand{\calG}{{\mathcal G}}
\newcommand{\calL}{{\mathcal L}}

\newcommand{\calT}{{\mathcal T}}
\newcommand{\calP}{{\mathcal P}}
\newcommand{\calX}{{\mathcal X}}

\newcommand{\calM}{{\mathcal M}}
\newcommand{\calW}{{\mathcal W}} 
\newcommand{\A}{\mathcal A}
\newcommand{\D}{\mathcal D}
\newcommand{\R}{\mathbb R}

\newcommand{\bK}{{\mathbf K}}
\newcommand{\bp}{{\mathbf p}}

\newcommand{\bu}{{\mathbf u}}
\newcommand{\bv}{{\mathbf v}}

\newcommand{\bw}{{\mathbf w}}
\newcommand{\ones}{{\mathbbm 1}}

\DeclareMathOperator*{\argmax}{arg\,max}
\DeclareMathOperator*{\argmin}{arg\,min}

\DeclareMathOperator*{\supp}{Supp}

\newcommand{\black}{\color{black}}

\newtheorem{theorem}{Theorem}

\newtheorem{problem}{Problem}
\newtheorem{remark}{Remark}

\newtheorem{lemma}{Lemma}
\newtheorem{definition}{Definition}
\newtheorem{proposition}{Proposition}

\makeatletter
\newenvironment{customproblem}[1]
  {\renewcommand{\theproblem}{#1}%
   \begin{problem}}
  {\end{problem}}
\makeatother

\newenvironment{keywords}{%
\vspace{1ex}\noindent\small
\textbf{Keywords:}\ }%
{\par\vspace{1ex}}

\title{\Large \bfseries Temporally Flexible Transport Scheduling on Networks with Departure-Arrival Constriction and Nodal Capacity Limits}

\author{%
Anqi Dong\thanks{Division of Decision and Control Systems and Department of Mathematics, KTH Royal Institute of Technology, SE-100 44 Stockholm, Sweden (\texttt{anqid@kth.se}).}%
\and Karl H. Johansson\thanks{Division of Decision and Control Systems and Digital Futures, KTH Royal Institute of Technology, SE-100 44 Stockholm, Sweden (\texttt{kallej@kth.se}).}%
\and Johan Karlsson\thanks{Department of Mathematics and Digital Futures, KTH Royal Institute of Technology, SE-100 44 Stockholm, Sweden (\texttt{johan.karlsson@math.kth.se}).}%
}

\begin{document}

\maketitle

\begin{abstract}
We investigate the optimal transport (OT) problem over networks, wherein supply and demand are conceptualized as temporal marginals governing departure rates of particles from source nodes and arrival rates at sink nodes. This setting extends the classical OT framework, where all mass is conventionally assumed to depart at $t = 0$ and arrive at $t = t_f$. Our generalization accommodates departures and arrivals at specified times, referred as \emph{departure--arrival} (DA) constraints. In particular, we impose nodal-temporal flux constraints at source and sink nodes, characterizing two distinct scenarios: (i) Independent DA constraints, where departure and arrival rates are prescribed independently, and (ii) Coupled DA constraints, where each particle’s transportation time span is explicitly specified. We establish that OT with independent DA constraints admits a multi-marginal optimal transport formulation, while the coupled DA case aligns with the unequal-dimensional OT framework. For line graphs, we analyze the existence and uniqueness of the solution path. For general graphs, we use a constructive path-based reduction and optimize over a prescribed set of paths.  From a computational perspective, we consider entropic regularization of the original problem to efficiently provide solutions based on multi-marginal Sinkhorn method, making use of the graphical structure of the cost to further improve scalability. Our numerical simulation further illustrates the linear convergence rate in terms of marginal violation. 
\end{abstract}


\section{Introduction}

Network flow problems aim to optimize the movement of commodities across a network while satisfying flow capacity and conservation constraints. These models have broad applications in mathematics, computer science, and transportation optimization~\cite{ahuja1988network}, with notable examples including the \emph{maximum-flow} and the \emph{minimum-cost} flow problems~\cite{dantzig2016linear}. The latter focuses on efficiently transporting goods or allocating resources (for instance, bandwidth and power) at minimal cost and plays a central role in the distribution of goods, information, and platoons~\cite{ekren2018constrained,haasler2024scalable}. In most of these classical formulations, time either does not appear explicitly or enters only through steady-state capacities rather than as a variable that can be controlled.

In many networked logistics and service systems, however, time is the primary degree of freedom. Control here is, in essence, a schedule: a measured release of mass, orchestrated in time, that guides it through junctions and grants it passage downstream. In port drayage, cranes, yard blocks, and gates operate under tight time-varying capacity. Export cutoffs couple departures to required arrivals, so planning becomes a joint choice of path and crossing times that keeps equipment busy without queues~\cite{steenken2004container}. In urban rail and bus transit, terminal departures and arrivals are prescribed, while headway and conflict windows impose capacity at stations and junctions. Operations hinge on pacing trains through dwell and merges, so that the timetable remains feasible~\cite{cacchiani2012nominal}. In data-center service chains, each request traverses firewall, cache, and inference under per-function throughput and end-to-end latency targets. Admission and timing at each hop prevent overload and deliver service objectives~\cite{mijumbi2015network}. These examples motivate a framework in which network transport cost, time-varying capacities, and departure–arrival specifications are unified in a \emph{scheduling-and-control} viewpoint.

The \emph{optimal transport} (OT) problem, originally proposed by Monge~\cite{monge1781memoire}, later formalized by Kantorovich~\cite{kantorovich1942translocation}, and more recently extended to dynamic formulations~\cite{benamou2000computational}, has attracted significant attention across mathematics~\cite{villani2009optimal}, control~\cite{chen2016optimal,chen2021optimal,eldesoukey2024schrodinger}, and computer science~\cite{peyre2019computational}. In particular, one considers two probability measures $\mu$ and $\nu$ and seeks a transportation plan that transfers mass from the former to the latter while minimizing a prescribed cost. The source distribution $\mu$ represents the initial mass configuration, and the target distribution $\nu$ specifies the demand.

The \emph{Monge} formulation~\cite{monge1781memoire}, with a quadratic cost in $\mathbb{R}^n$, seeks a transport map $T: x \mapsto y = T(x)$ that minimizes the cost functional
\begin{align}
J(T) := \int_{\mathbb{R}^n} \|x-T(x)\|_2^2 \; \mu(dx),    
\end{align}
among admissible maps $T$ that push $\mu$ onto $\nu$. However, such a map may not always exist. The \emph{Kantorovich} relaxation~\cite{kantorovich1942translocation} reformulates the problem as an optimization over couplings between $\mu$ and $\nu$. Instead of a deterministic transport map, it seeks a probability measure $\pi$ on $\mathbb{R}^n \times \mathbb{R}^n$ that has $\mu$ and $\nu$ as its marginals and minimizes
\begin{align}
J(\pi) := \int_{\mathbb{R}^n \times \mathbb{R}^n} \|x-y\|_2^2 \; \pi(dx, dy).  
\end{align}
This static view already captures the geometry of redistributing mass. Dynamic formulations interpret the same problem as continuous-time motion that moves $\mu$ to $\nu$ over a finite time horizon~\cite{benamou2000computational}, making the role of timing more explicit.

A natural extension of OT relaxes the marginal constraints in the Kantorovich formulation and instead imposes upper and lower bounds on the marginal distributions or their cumulative functions~\cite{rachev2006mass}.
Similarly, constraints can be imposed on the coupling itself, yielding a bi-marginal constraint of the form $\underline{r} \leq \pi(dx,dy) \leq \overline{r}$.
Physical constraints along transport paths naturally arise in applications. Moment-type constraints were introduced in~\cite[Section 4.6.3]{rachev1998mass} and later generalized in~\cite{ekren2018constrained}. Congestion effects as well as origin destination constraints, which can impede transport, have also been studied \cite{haasler2020optimal,haasler2024scalable}.
 A path-dependent cost to mitigate congestion was considered in~\cite{carlier2008optimal}, while~\cite[Section 4]{santambrogio2015optimal} provides a comprehensive review of this topic. Constraints on probability densities have also been explored to model capacity limitations of transportation media~\cite{korman2013insights} or dynamical flow constraints~\cite{gladbach2022limits}. In addition, our recent work~\cite{stephanovitch2024optimal} focus on a specific \emph{flow-rate} (also, \emph{momentum}) constraint imposed on momentum in a dynamic OT formulation,
$
\underline{r} \leq \rho(t, x)v(t, x) \leq \overline{r},
$
which models a toll station. The same flow-rate bound was then reformulated within the multi-marginal Kantorovich framework~\cite{dong2024monge} by introducing a “crossing time’’ as an additional marginal over the time variable, thereby the coupling jointly encodes where mass travels and when it passes each node.


The present work builds on this line of thought and brings it into a network setting where time is the control input. We tie optimization and control by taking the schedule as the decision. Path-wise couplings and crossing-time profiles are chosen to minimize transport cost under marginal and capacity constraints. A schedule is a distribution of crossing times along each path. Boundary time profiles act as release and acceptance laws at the source and the sink. At each interior node, a single crossing-time profile meters flow and enforces conservation, since the same time index labels both incoming and outgoing mass under the capacity limit. From a network optimization perspective, our formulation can be viewed as minimum-cost flow problem with time flexibility. Classical minimum-cost flow fixes edge capacities and costs and optimizes how much mass uses each arc, while time either does not appear or is treated as a static resource. Here, the network and total transported mass are fixed, and time itself becomes control variable. The decision is when mass is allowed to traverse each node and edge, under explicit departure–arrival profiles and nodal throughput limits. In this sense, we lift minimum-cost flow from a static allocation problem to a path-wise scheduling problem while remaining within an optimal transport framework.

Herein, we cast optimal transport on networks as a \emph{scheduling-and-control} problem with time as the actuator. The framework is path-wise. Flows are represented by multi-dimensional couplings whose marginals encode source and sink distributions together with crossing-time profiles at interior nodes. In this way, release curves, delivery deadlines, and per-node throughput constraints are handled within a single optimal transport formulation rather than through ad hoc scheduling rules. To the best of the authors’ knowledge, this is the first framework that integrates departure–arrival specifications and time-varying nodal capacities into a path-wise optimal transport model on networks, provides structural conditions for well-posedness, and propose scalable \emph{Schr\"odinger-Fortet-Sinkhorn} algorithm. The result is a unified description of network flow, timing, and capacity that fits naturally with boundary control, state constraints, and receding-horizon updates. We impose departure–arrival (DA) rationales in two regimes:
\begin{enumerate}
\item[i)] \emph{Independent DA constraints:} departure and arrival rates are prescribed separately as boundary controls. The schedule reshapes crossing times in order to satisfy both boundary laws and interior capacity limits. This setting is developed in Section~\ref{sec:ind} through Problems~{\bf \ref{prob:problem1}} and~\ref{prob:prob1-prime} and leads to a multi-marginal optimal transport formulation with an additional time marginal at each node.
\item[ii)] \emph{Coupled DA constraints:} each particle’s transport duration is fixed, so the controller chooses a start time and the arrival time then follows deterministically. This regime is formulated in Section~\ref{sec:couple} as Problem~{\bf \ref{prob:prob2}} and \ref{prob:prob2-prime}. It yields an unequal-dimensional optimal transport problem with a hard link between release and delivery, suited to timetable and deadline-driven settings.
\end{enumerate}
Nodal flow-rate constraints act as state and input limits along each path and are enforced through cost terms and inequality constraints on crossing times. On line graphs, we provide feasibility and existence guarantees in Lemma~\ref{lemma:feasible-order} and Proposition~\ref{prop:existence}, and uniqueness guarantees in Theorems~\ref{thm:unique1} and~\ref{thm:unique1-prime}. In the independent case, the cost satisfies a generalized Monge condition, while in the coupled case, it satisfies an $x$-twist condition as shown in Proposition~\ref{prop:pureness} and Theorem~\ref{thm:unique-coupled-line}. These properties identify when a well-posed schedule exists and when it is uniquely determined by the DA and capacity specifications, and form the analytical backbone for the path-wise generalization and the entropic Sinkhorn-type algorithms that we develop later for large network instances.

For general graphs, we introduce a node-aggregation reduction that preserves the transport dynamics while simplifying the topology. Nodes are collapsed into effective groups so that DA-constrained transport on large graphs can be analyzed and computed through lower-dimensional surrogates without losing the structure of capacities and schedules, as formalized in Problem~\ref{prob:problem3}. This reduction keeps path-wise timing, nodal capacities, and DA specifications intact and therefore serves as a bridge between realistic network layouts and tractable line-graph models. On top of this reduction, we develop scalable algorithms based on entropic regularization and a graph-structured Sinkhorn scheme with shared nodal multipliers, summarized in Algorithm~\ref{alg:sinkhorn_graph}. The shared multipliers act as time-indexed node prices that coordinate flow across all paths passing through a node and enforce both boundary and capacity constraints. The resulting alternating projection scheme performs explicit projections onto departure and arrival profiles and onto nodal throughput bounds, enabling large-scale structured transport on networks with explicit DA specifications and node-wise capacity limits. This scheme's linear convergence rate is established in Theorem~\ref{thm:convergence} and confirmed by numerical experiments in Section~\ref{sec:num}.

The rest of the paper is organized as follows. Section~\ref{sec:prelim} formalizes the problem setup. Section~\ref{sec:ind} discusses the case where DA constraints are independent, and Section~\ref{sec:couple} considers the coupled DA setting. In Section~\ref{sec:general}, we extend the problem from single-source and single-sink transportation over line graphs to general graphs and present node aggregation methods and a Sinkhorn algorithm tailored to the structure of the cost functionals. Numerical simulations are presented in Section~\ref{sec:num}. Finally, Section~\ref{sec:conclusion} revisits the scheduling viewpoint, summarizes our theoretical and algorithmic contributions.

\section{Problem setup and background}\label{sec:prelim}

We start with a recap of optimal transport, focusing on its extension from bi-marginal to multi-marginal formulations, marginal- and coupling-wise relaxations, and the incorporation of physical constraints. We then specialize to network-based optimal transport, which will serve as the reference point for the temporally flexible formulations in Sections~\ref{sec:ind} and~\ref{sec:couple}.

\subsection{Multi-marginal optimal transport}

\emph{Multi-marginal optimal transport} (MMOT) generalizes the classical optimal transport framework to settings with more than two marginals, and seeks a joint distribution that minimizes a high-dimensional cost while matching given marginals. Key challenges include establishing the existence and uniqueness of solutions for general transport costs. See, for instance, \cite{chiappori2016multidimensional,kim2014general,pass2015multi}.

Given compactly supported non-negative measures $\mu^0, \mu^1, \dots, \mu^{\calT}$ on spaces $\mathbb R^n$ and a continuous cost function $c(x_0,x_1,\dots,x_\calT)$, the Monge formulation of MMOT seeks transport maps $T_{\ell}: \R^n \to \R^n$ for all $\ell = 1, \dots, \calT$ that minimize
\begin{align*}
\int_{\R^n} c\big(x_0,T_1(x_0),\dots,T_\calT(x_0)\big)\, d\mu^0(x_0),
\end{align*}
subject to $T_\ell{}_{\sharp}\mu^0 = \mu^\ell$ for $\ell = 1,\dots,\calT$. As usual, the push-forward notation $T_\ell{}_{\sharp}\mu^0 = \mu^\ell$ means that for any Borel set $S \subset \R^n$, $\mu^0\big(T_\ell^{-1}(S)\big) = \mu^\ell(S)$. As in the bi-marginal case, a Monge solution may fail to exist.

The Kantorovich formulation of MMOT replaces transport maps by couplings and minimizes
\begin{align}\label{eq:multi-marginal-obj}
\langle c,\pi \rangle
:= \int_{\R^n\times \dots\times \R^n}
c(x_0,\dots,x_\calT)\, d\pi (x_0,\dots,x_\calT),
\end{align}
over the set of admissible couplings
\begin{align}\label{eq:multi-marginal-set}
\Pi(\mu^0,\dots,\mu^\calT)
=\Big \{\pi \;\Big|\; P_\ell(\pi) = \mu^\ell,\ \ell = 0,1,\dots,\calT \Big\},    
\end{align}
and $P_\ell(\pi)$ denotes the projection of $\pi$ onto its $\ell$-th marginal, i.e.,
\begin{align*}
P_\ell(\pi)
:= \int_{\R^n\times \dots \times\R^n}
d\pi(x_0,\dots,x_{\ell-1},x_\ell,x_{\ell+1},\dots,x_{\calT}),
\end{align*}
where the iterated integral is taken over all variables $x_k$ with $k\neq \ell$.


\subsection{Conditions for uniqueness of the solution}
We next recall classical conditions that guarantee the uniqueness of the optimal coupling in the bi-marginal setting, and then their multi-marginal analogue.

\begin{definition}[{\it \bf Monge condition}~\cite{rachev1998mass}]\label{def:monge}
A cost function $c:\mathbb{R}^2 \to \mathbb{R}$ is said to satisfy the Monge condition\footnote{Also, the Spence--Mirrlees condition in economics \cite{chiappori2016multidimensional,santambrogio2015optimal}} if, for real variables $x,y \in \mathbb{R}$, its cross difference
$$
\Delta_{(x,x')}^{(y,y')}(c)
:= c(x,y) + c(x',y') - c(x,y') - c(x',y)
$$
satisfies $\Delta_{(x,x')}^{(y,y')}(c) \leq 0$ for all $x' \geq x$ and $y' \geq y$.
\end{definition}


If $c(x,y)$ is quasi-antitone, in the sense that $\Delta_{(x,x')}^{(y,y')}(c) \leq 0$, or quasi-monotone, in the sense that $\Delta_{(x,x')}^{(y,y')}(c) \geq 0$, then the minimizer of the transport problem with cost $c$ is unique, see \cite{cambanis1976inequalities} and \cite[Theorem 3.1.2]{rachev1998mass}. In matching theory, the cross-difference condition expresses that, for given matching pairs $(x,y)$ and $(x', y')$, swapping partners to $(x,y')$ and $(x',y)$ cannot reduce the total cost. This rules out profitable swaps and leads to an optimal and stable matching. The link between cross differences and second derivatives is made precise through a Taylor expansion argument, as discussed in \cite{cambanis1976inequalities,carlier2003class,chiappori2016multidimensional} and extended to higher dimensions in \cite{pass2012local}.

\begin{definition}[{\it \bf Twist condition}~\cite{villani2009optimal}]
Let $c(x,y)\in C^1$. The cost $c$ is said to satisfy the twist condition if, for each fixed $x$, the map $y \mapsto \nabla_x c(x,y)$ is injective, and for each fixed $y$, the map $x \mapsto \nabla_y c(x,y)$ is injective. Equivalently, for all distinct pairs $y \neq y'$ and $x \neq x'$, one has
$D_x c(x, y) \neq D_x c(x, y')$ and 
$D_y c(x, y) \neq D_y c(x', y)$.
\end{definition}
In other words, the gradient $\nabla_x c(x,y)$ uniquely determines $y$ and $\nabla_y c(x,y)$ uniquely determines $x$. This implies that optimal couplings are induced by transport maps and are therefore unique. The twist condition is thus a convenient sufficient criterion for the uniqueness of the optimal transport plan.

In the multi-marginal setting, an analogous structure is obtained by imposing a Monge-type condition on all two-marginal projections of the cost.

\begin{definition}[{\it \bf Generalized Monge condition} \cite{rachev1998mass}]\label{def:genmonge}
Consider a cost functional $c:\mathbb{R}^n \to \mathbb{R}$ of the form $c(x_1,\dots,x_n)$, where $(x_1,\dots,x_n)$ lie in the support of a measure on $\calX^n$ with $\calX \subset \mathbb{R}$. The cost $c$ is said to satisfy the generalized Monge condition if, for any pair of indices $1 \leq i < j \leq n$ and all $x_i \leq x_i'$ and $x_j \leq x_j'$, one has $c(x_i,x_j) + c(x_i',x_j')
\leq c(x_i,x_j') + c(x_i',x_j)$,
with all other coordinates held fixed.
\end{definition}

Under the generalized Monge condition, one can show that, for each marginal $\mu_i$, the optimal coupling takes the form
\begin{align*}
\pi = \big(T^1,\dots,T^{i-1},\mathrm{Id},T^{i+1},\dots,T^n\big)_{\sharp}\mu_i,
\end{align*}
where each transport map $T^j$ is unique and monotone for all $j = 1,\dots,n$, and this representation holds for every $i = 1,\dots,n$. Hence, the joint coupling $\pi$ is uniquely determined across all marginals. This generalized Monge structure will reappear later when we establish uniqueness in our one-dimensional line-graph setting.

\subsection{Problem: optimal transport on networks}

We established departure--arrival constrained models over a graph based on the Kantorovich formulation. Let $\calG=(\calV,\calE,\calW)$ be a directed graph, where $\calV$ is the node set, $\calE$ the directed edge set, and $\calW$ the collection of edge weights. Each edge $e\in\calE$ carries a weight $w(e)$ that represents the cost (or, distance) of transporting one unit of mass along that edge. A subset $\calV^- \subseteq \calV$ collects source nodes and a subset $\calV^+ \subseteq \calV$ collects sink nodes.

Discrete source and sink distributions are given by nonnegative functions $\mu$ and $\nu$ on $\calV$, supported on $\calV^-$ and $\calV^+$ respectively. We assume that total supply matches total demand, i.e.,
$\sum_{v\in\calV}\mu(v)=\sum_{v\in\calV}\nu(v)$.
A transportation plan is a nonnegative function $\pi$ on $\calV\times\calV$, where $\pi(v_i,v_j)$ specifies how much mass is sent from $v_i$ to $v_j$ and consistency with marginals requires $\sum_{v_i\in\calV}\pi(v_i,v_j)=\nu(v_j)$ and $\sum_{v_j\in\calV}\pi(v_i,v_j)=\mu(v_i)$.
The transportation cost is encoded by a function $c:\calV\times\calV\to\mathbb R$, typically induced by the graph geometry, for instance by the length or travel time of the shortest path from $v_i$ to $v_j$ in $\calG$. The network-based optimal transport problem then minimizes
$\sum_{v_i,v_j\in\calV} c(v_i,v_j)\,\pi(v_i,v_j)$,
over all plans $\pi$ that satisfy the marginal constraints. This coincides with a minimum-cost flow problem \cite{ahuja1988network} on $\calG$ with node supplies $\mu$ and demands $\nu$. In this static setting, time does not appear explicitly. In the next sections, we enrich the model by attaching time coordinates to paths, so that departures, crossings, and arrivals become time profiles rather than instantaneous transfers, and by imposing capacity constraints at intermediate nodes together with departure–arrival specifications at the boundary.

\section{Independent DA constraints: problem formulation and uniqueness}\label{sec:ind}
In this section, we specialize to a path graph and make time explicit. We consider a graph $\calG = (\calV,\calE,\calW)$ whose vertices are ordered as
$\bp = \{v_0,v_1,\dots,v_{\calT-1},v_{\calT}\}$,
with $v_0$ the unique source and $v_{\calT}$ the unique sink. Each edge $e = (v_{i-1},v_i) \in \calE$ is assigned with a positive weight $w(v_{i-1},v_i) > 0$. We assume that one unit of mass is dispatched from $v_0$ and must reach $v_{\calT}$, so that in the static setting, transport takes place along the single path $\bp$ and reduces to a shortest-path problem. We consider a departure-arrival over a line graph, in which departures from $v_0$, crossings at the intermediate nodes $v_1,\dots,v_{\calT-1}$, and arrivals at $v_{\calT}$ are described by time profiles subject to departure–arrival specifications and nodal capacity bounds.

The mass flow is governed by departure and arrival rate constraints. The departure rate $\mu^0(t)$ describes the mass density over time at which mass leaves the source node $v_0$, so that $\int_t^{t+dt }\mu^0(\tau)\,d\tau$ is the amount departing in the interval $[t,t+dt]$. Similarly, the arrival rate $\mu^\calT(t)$ describes the mass density over time at which mass reaches the sink node $v_\calT$, with $\mu^\calT(t)\,dt$ the amount arriving in $[t,t+dt]$. In addition to these boundary specifications, the flow is subject to rate bounds at each intermediate node along the path,
\[
\sigma^{i}(t) \leq r(v_i),
\qquad i = 1,\dots,\calT-1,
\]
where $\sigma^{i}(t)$ denotes the instantaneous rate of mass crossing node $v_i$. These constraints ensure that the throughput at each intermediate node does not exceed the prescribed flux limit, reflecting physical or logistical capacity. An illustration is given in Fig.~\ref{fig:single-path}.

\begin{figure}[htb!]
\centering
\includegraphics[width=.7\linewidth]{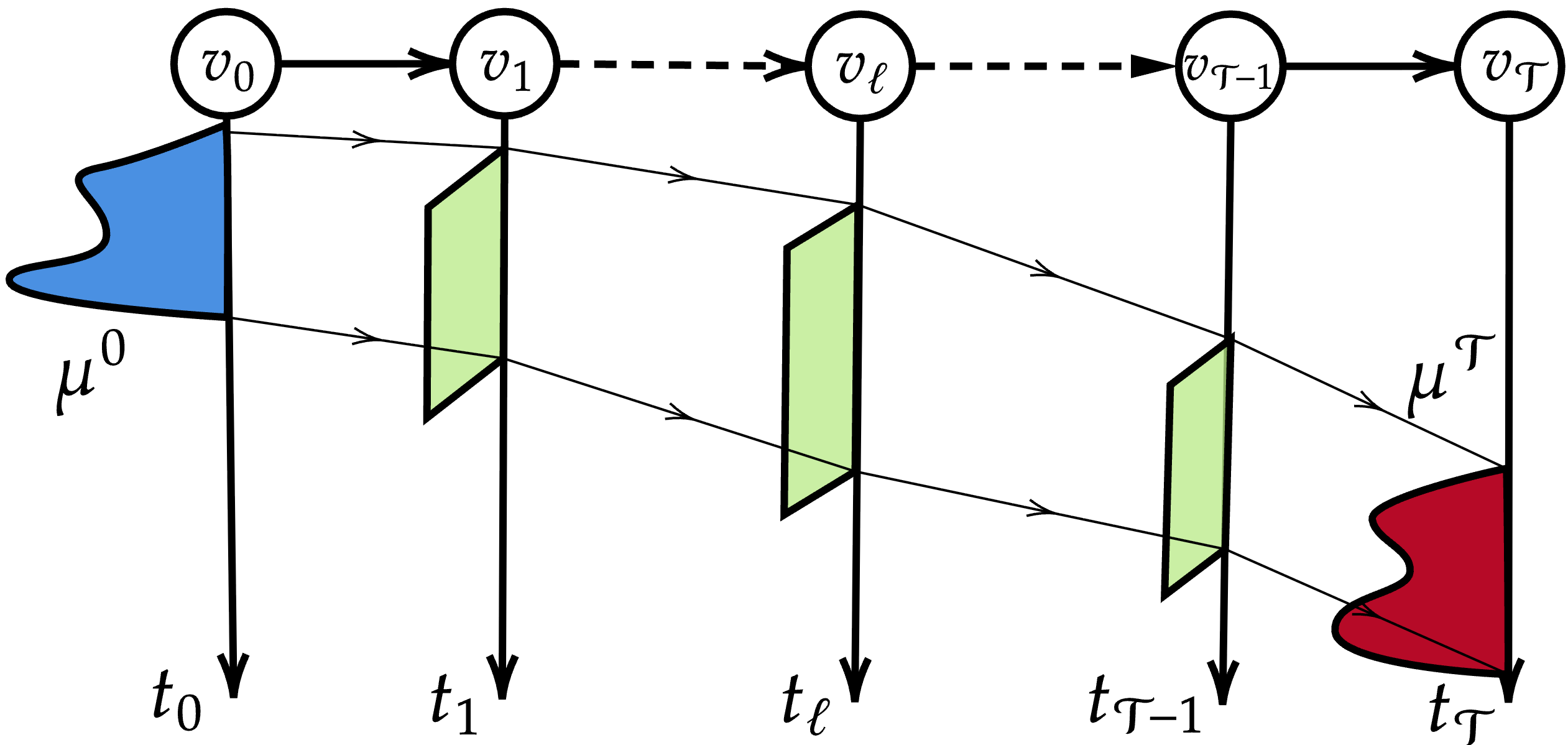}
\caption{Mass flow along the line graph $\bp=\{v_0,v_1,\dots,v_{\calT-1},v_\calT\}$. Each particle departs from $v_0$ according to the departure rate $\mu^0$, travels sequentially through intermediate nodes $v_1, v_2,\dots,v_{\calT-1}$ abides flow-rate constraints at each node, and arrives at the sink $v_\calT$ according to the arrival rate $\mu^\calT$.}
\label{fig:single-path}
\end{figure}

Independent DA constraints refer to the case where the departure distribution $\mu^0(t)$ and the arrival distribution $\mu^\calT(t)$ are prescribed separately, without specifying how individual departures are paired with arrivals. Our first step is to formulate the scheduling problem on a single intermediate node. This will serve as the basic building block for the multi-node line graph. We consider a single toll located at $v_1$, which imposes a flow-rate constraint $\sigma^1(t) \leq r$ with threshold $r>0$. The times at which mass departs from $v_0$, crosses $v_1$, and arrives at $v_2$ are denoted by $t_0$, $t_1$, and $t_2$, respectively, and the cost of transporting mass along the two segments is modeled by
\begin{align}\label{eq:cost-function-1}
c(t_0, t_1, t_2)
=
\frac{w(v_0, v_1)}{t_1 - t_0}
+
\frac{w(v_1, v_2)}{t_2 - t_1}.
\end{align}
The first term measures the transit cost from $v_0$ to $v_1$ over the time interval $t_1 - t_0$, and the second term the cost from $v_1$ to $v_\calT$ over $t_2 - t_1$. By optimizing this cost under the DA and capacity constraints, we obtain a temporally feasible transport plan that respects both spatial and temporal structure, as formalized below.

\begin{problem}\label{prob:problem1}
{\bf \textit{(Independent DA with single intermediate node)}} Given probability measures $\mu^0(t_0)$ and $\mu^\calT(t_2)$, along with a flow-rate bound $r > 0$, determine a probability measure $\pi(t_0,t_1,t_2)$ on $\mathcal{M}_+([0, t_f]^3)$ that minimizes
\begin{subequations}
\begin{align}\label{eq:obj1}
\iiint_{t_0,t_1,t_2} c(t_0, t_1, t_2) \,\pi(dt_0,dt_1,dt_2),
\end{align}
subject to the departure and arrival constraints
\begin{align}
&\iint_{t_1,t_2} \pi(dt_0,dt_1,dt_2) = \mu^0(t_0),\\
&\iint_{t_0,t_1} \pi(dt_0,dt_1,dt_2) = \mu^\calT(t_2),
\end{align}
and the flow-rate constraint
{\setlength{\abovedisplayskip}{7pt}%
 \setlength{\belowdisplayskip}{7pt}%
\begin{align}
\iint_{t_0,t_2} \pi(dt_0,dt_1,dt_2) := \sigma^1(dt_1) \leq r dt_1.
\end{align}}%
\end{subequations}
\end{problem}
In what follows, we first characterize when a feasible schedule exists (in terms of an ordering relation between $\mu^0$ and $\mu^\calT$), then establish the existence of minimizers for suitable bounds $r$, and finally prove uniqueness of the optimal coupling by exploiting the Monge and generalized Monge structures introduced in Section~\ref{sec:prelim}.

\subsection{Feasibility}

We first characterize when a pair of independent DA profiles is compatible with a minimum travel time.

\begin{lemma}[\bf \textit{Feasibility of independent DA}]\label{lemma:feasible-order}
With fixed minimum travel time $\Delta>0$, a pair $(\mu^0,\mu^{\calT})$ on $[0,t_f]$ with
$\int_0^{t_f}\mu^0(dt_0)=\int_0^{t_f}\mu^{\calT}(dt_{\calT})=1$
is feasible if and only if $\mu^0$ is dominated by the $\Delta$-shift of $\mu^{\calT}$ in first-order (stochastic) order~\cite[Chapter 1]{shaked2007stochastic}, that is,
\[
F_{0}(t)\ \ge\ F_{\calT}(t+\Delta)\quad\text{for all } \;\; t\in[0,t_f],
\]
where $F_{0}(t)=\int_0^{t}\mu^0(ds)$ and $F_{\calT}(t)=\int_0^{t}\mu^{\calT}(d\tau)$. Equivalently, for every bounded increasing test function $\phi\colon\mathbb R\to\mathbb R$, we have
\[
\int_0^{t_f}\phi(t)\,\mu^0(dt)\ \le\ \int_0^{t_f}\phi(t-\Delta)\,\mu^{\calT}(dt),
\]
with $\phi(t-\Delta)$ understood as $\phi(0)$ when $t<\Delta$ \cite[Eq. 1.A.7]{shaked2007stochastic}. In particular, we have\footnote{We write $\supp(\cdot)$ for the support of a measure.}
\begin{align*}
\begin{aligned}
\inf\supp(\mu^{\calT})\ &\ge\ \inf\supp(\mu^0)+\Delta,\\
\sup\supp(\mu^{\calT})\ &\ge\ \sup\supp(\mu^0)+\Delta.    
\end{aligned}   
\end{align*}
\end{lemma}

\begin{proof}
For any feasible coupling with $t_{\calT}-t_0\ge\Delta$ almost surely, the event inclusion $\{t_{\calT}\le t+\Delta\}\subseteq\{t_0\le t\}$ holds for every $t$. Taking probabilities gives $F_{\calT}(t+\Delta)\le F_{0}(t)$ for all $t$. Assume $F_{0}(t)\ge F_{\calT}(t+\Delta)$ for all $t$. Define the left-continuous generalized inverses $F_{0}^{-1}(u)=\inf\{t\mid F_{0}(t)\ge u\}$, and $F_{\calT}^{-1}(u)=\inf\{t\mid F_{\calT}(t)\ge u\}$. Then 
$$
F_{\calT}^{-1}(u)\ \ge\ F_{0}^{-1}(u)+\Delta
$$
for all $u\in(0,1)$. Define $\pi$ as the image of the Lebesgue measure on $(0,1)$ under the map
$u\ \mapsto\ \big(F_{0}^{-1}(u),\,F_{\calT}^{-1}(u)\big)$.
By construction, the first marginal of $\pi$ has CDF $F_{0}$ and the second marginal has CDF $F_{\calT}$. The pointwise inequality above gives $t_{\calT}\ge t_0+\Delta$ for $\pi$-almost every pair, hence $\pi$ is feasible. The support bounds follow by evaluating $F_{0}(t)\ge F_{\calT}(t+\Delta)$ at $t=\inf\supp(\mu^{\calT})-\Delta$ and at $t=\sup\supp(\mu^{0})-\Delta$.
\end{proof}

The condition $F_{0}(t)\ge F_{\calT}(t+\Delta)$
is related to order-type comparisons between distributions
(see, e.g.,~\cite{hadar1969rules,shaked2007stochastic}),
but unlike in martingale optimal transport \cite{beiglbock2016problem} we do not impose any
martingale or mean-preserving constraint here. We only use it as
a simple feasibility check for the DA pair.

\subsection{Existence and uniqueness for single intermediate node}
Before analyzing uniqueness, we first establish the existence of a minimizer for Problem~\ref{prob:problem1}.

\begin{proposition}[\bf \textit{Existence of solution}]\label{prop:existence}
For sufficiently large $r$, Problem~\ref{prob:problem1} admits a minimizing solution $\pi$.
\end{proposition}

\begin{proof}
The space of admissible measures $\pi$, namely
\begin{align}
\begin{aligned}
&\Pi = \bigg\{\pi \in \mathcal{M}_+\big([0, t_f]^3\big) \Big | \; \iint_{t_1, t_2} \!\!\!\!\!\! \pi = \mu^0(t_0),\ \iint_{t_0, t_1} \!\!\!\!\!\! \pi = \mu^\calT(t_2), \iint_{t_0, t_2} \!\!\!\!\!\! \pi:= \sigma^1(t_1) \leq r \, dt_1 \bigg\},  
\end{aligned}
\end{align}
is tight and closed for narrow convergence. Hence it is weakly compact. Since the cost \eqref{eq:obj1} is lower semicontinuous, the direct method gives existence of a minimizer once we know that $\Pi$ is non-empty.

We now show that $\Pi$ is non-empty when $r$ is large enough by an explicit construction.
By Lemma~\ref{lemma:feasible-order}, the minimum travel time condition is equivalent to
\begin{align}\label{eq:quantile-gap}
F_{\calT}^{-1}(u)\ge F_0^{-1}(u)+\Delta
\qquad
\text{for all }u\in(0,1).
\end{align}
Fix $\varepsilon\in(0,\Delta)$ and take sufficient large $r$ so that $\varepsilon+1/r<\Delta$.
Let $u\sim \mathrm{Unif}(0,1)$ and $s\sim \mathrm{Unif}(0,1/r)$ be independent, and define
$t_0 = F_0^{-1}(u)$,  $t_2 = F_{\calT}^{-1}(u)$, and $t_1 = t_0+\varepsilon+s$. Let $\tilde\pi$ be the law of $(t_0,t_1,t_2)$.
Then $\tilde\pi$ has $t_0$- and $t_2$-marginals $\mu^0$ and $\mu^{\calT}$. Moreover, \eqref{eq:quantile-gap} and $s\le 1/r$ imply that $t_1-t_0\ge \varepsilon$ and $t_2-t_1 \ge \Delta-\varepsilon-\frac{1}{r} > 0$, so both traveling times are positive. The condition $\varepsilon+1/r<\Delta$ means $\epsilon$ plus the maximal delay $1/r$ still fits inside the available travel-time budget $\Delta$. Here, $t_1$ is the crossing time at intermediate node. The constant $\varepsilon$ is a fixed transit delay after departure, before any mass can reach the node. The extra waiting time $s$, sampled uniformly on an interval of length $1/r$, spreads each departure cohort over a window of width $1/r$ so that the flow respects the capacity bound $r$.

Finally, conditional on $t_0$, the variable $t_1$ is uniform on an interval of length $1/r$. Hence, for any Borel set $S\subset[0,t_f]$, we have $\mathbb P(t_1\in S\mid t_0)\le r\,|S|$ since a uniform law on an interval of length $1/r$ has density $r$.
Taking expectations yields to $\sigma^1(S)=\mathbb P(t_1\in S)\le r\,|S|$, which is equivalent to $\sigma^1\le r\,dt_1$. Therefore, there always exists a $\tilde\pi$ so that $\tilde\pi\in\Pi$, and $\Pi$ is non-empty.
\end{proof}

We are now in a position to establish the uniqueness of the optimizer for Problem~\ref{prob:problem1}. For simplicity, we omit the superscript and adopt the notation $\sigma^{1} \equiv \sigma$ for the remainder of this section.

\begin{lemma}[\bf \textit{Monotonicity}]\label{lemma:monotone}
Consider the cost function
$$
c_{t_0,t_1} := \frac{w(v_0, v_1)}{t_1 - t_0},
$$
where $(t_0,t_1) \in \mathcal{M}_+\big([0, t_f]^2\big)$, and let $\mu^0$ and $\sigma$ be measures supported on $[0,t_f]$. Assume that $\sigma(t)\,dt$ is absolutely continuous with respect to the Lebesgue measure. Then, the minimizer of the Kantorovich problem
$$
\min_{\pi} \int_{t_0} \int_{t_1} c_{t_0,t_1} \,\pi(dt_0,dt_1),
$$
where $\pi$ represents a coupling of the marginals $\mu^0(dt_0)$ and $\sigma(t_1)\,dt_1$, is unique and supported on the graph of a non-increasing function $T^{\mathcal D}(t_1)$.
\end{lemma}

\begin{proof}
We first observe that for any two pairs $(t_1,t'_1)$ and $(t_0,t'_0)$ satisfying
$t_f > t_1 > t'_1 > t_0 > t'_0 > 0$,
it holds that
\begin{align*}
\frac{w(v_0,v_1)}{t_1 - t_0} + \frac{w(v_0,v_1)}{t'_1 - t'_0} > \frac{w(v_0,v_1)}{t'_1 - t_0} + \frac{w(v_0,v_1)}{t_1 - t'_0}.
\end{align*}
The ordering in this inequality characterizes $c_{t_0,t_1}$ as being quasi-monotone, in the terminology of \cite{cambanis1976inequalities}, see also \cite[Section 3.1]{rachev1998mass}.

It follows from \cite[Theorem 2.18 and Remark 2.19]{villani2021topics} that the optimal coupling $\pi$ exists and is given by the monotone rearrangement of $\mu^0$ and $\sigma$. That is, for a suitable function $T^{\mathcal D}(t_1)$, the optimal transport plan satisfies $ \int_{0}^{T^{\mathcal D}(t_1)}\mu^0(dt_0) =  \int_0^{t_1} \sigma(s)\,ds$.
Since $T^{\mathcal D}(t_1) < t_1$, it follows that $T^{\mathcal D}(t_1)$ is non-increasing  (and remains constant over time intervals corresponding to Dirac components of $\mu^0$, and also, over intervals where $\sigma$ carries no mass).
\end{proof}

In the same spirit, the cost $c_{t_1,t_2}$ with $(t_1,t_2) \in \mathcal M_+([0, t_f]^2)$ is quasi-monotone. Consequently, for a suitable function $T^{\mathcal A}(t_1)$, the optimal transport plan satisfies
$\int_0^{T^{\mathcal A}(t_1)} \mu^\calT(dt_2) =  \int_0^{t_1} \sigma(s)\,ds$
for $T^{\mathcal A}(t_1)>t_1$, so that $T^\mathcal{A}(t_1)$ is non-increasing. 

With the monotonicity of the push-forward maps $T^{\mathcal{D}}(t_1)$ and $T^{\mathcal{A}}(t_1)$ established, we now present the following uniqueness result.

\begin{theorem}[\bf \textit{Uniqueness}]\label{thm:unique1}
Under the assumptions of Problem~\ref{prob:problem1}, the minimizer is unique. Moreover, there exist functions $T^\D(t_1)$ and $T^\A(t_1)$ that are monotonically non-increasing such that $\pi = \big(T^\D, {\rm Id, T^\A}\big)_{\sharp}\sigma$,
where ${\rm Id}(t_1) = t_1$ is the identity map, and $\sigma(t_1)\,dt_1$ is an absolutely continuous measure on $[0,t_f]$ satisfying $\sigma(t_1) \leq r$.
\end{theorem}

\begin{proof}
Let $\pi$ be a minimizer as established in Proposition~\ref{prop:existence}, and define $\pi_{t_0,t_1} := \int_{t_2} \pi$, $\pi_{t_1,t_2} := \int_{t_0} \pi$, and $\sigma := \iint_{t_0,t_2} \pi$. Since the objective function satisfies
\begin{align*}
\iiint_{t_0,t_1,t_2} \!\!\!\!\!\!\!\! c \, \pi
= \iint_{t_0,t_1} \!\!\!\!\! c_{t_0,t_1} \pi_{t_0,t_1} +
\iint_{t_1,t_2} \!\!\!\!\!\!\! c_{t_1,t_2} \pi_{t_1,t_2},
\end{align*}
it follows that both $\pi_{t_0,t_1}$ and $\pi_{t_1,t_2}$ must be minimizers of their respective problems. Otherwise, there would exist alternative couplings $\hat\pi_{t_0,t_1}$ and $\hat\pi_{t_1,t_2}$ achieving strictly lower costs
\begin{align*}
\iint_{t_0,t_1} c_{t_0,t_1} \hat\pi_{t_0,t_1} < \iint_{t_0,t_1} c_{t_0,t_1} \pi_{t_0,t_1}, \ \iint_{t_1,t_2} c_{t_1,t_2} \hat\pi_{t_1,t_2} < \iint_{t_1,t_2} c_{t_1,t_2} \pi_{t_1,t_2}.
\end{align*}
Since both couplings share the same marginal on the $t_1$-axis, we have
\begin{align*}
\int_{t_0} \hat\pi_{t_0,t_1}(dt_0, dt_1) = \int_{t_2} \hat\pi_{t_1,t_2}(dt_1, dt_2) = \sigma(t_1)\,dt_1.
\end{align*}
By the gluing lemma \cite[page 11]{villani2009optimal}, there exists a coupling $\hat\pi$ on $\mathcal{M}_+([0, t_f]^3)$ that preserves these marginals while achieving a strictly lower cost, contradicting optimality. Thus, both $\pi_{t_0,t_1}$ and $\pi_{t_1,t_2}$ must be optimal.

Next, we show that the relaxed marginal $\sigma$ is unique, which establishes the desired result via Lemma~\ref{lemma:monotone}. Suppose there exist two distinct minimizers, $\pi^a$ and $\pi^b$. Each induces a density $\sigma^a(t_1)$ and $\sigma^b(t_1)$, respectively, on the $t_1$-axis, along with corresponding transport maps $(T^{\D,a}, T^{\A,a})$ and $(T^{\D,b}, T^{\A,b})$. Since both $\sigma^a(t_1)$ and $\sigma^b(t_1)$ satisfy the constraint $\sigma(t_1) \leq r$, their convex combination $\bar \sigma = \frac{1}{2} \sigma^a + \frac{1}{2} \sigma^b$
also remains feasible. Similarly, the convex combination
$\bar \pi = \frac{1}{2} \pi^a + \frac{1}{2} \pi^b$
defines an optimal coupling. However, the bi-marginal coupling
\begin{align*}
\bar \pi_{xt} = (T^{\D,a}, {\rm Id})_{\sharp}\tfrac{1}{2} \sigma^a +
(T^{\D,b}, {\rm Id})_{\sharp}\tfrac{1}{2} \sigma^b
\end{align*}
fails to be supported on the graph of a function, which implies that there exists a \emph{cheaper} coupling with strictly lower cost, obtained by monotone rearrangement of $\bar \pi_{xt}$. Since $\bar \pi_{xt}$ satisfies the marginal constraints, meaning
$T^{\D,a}_{\sharp} \sigma^a + T^{\D,b}_{\sharp} \sigma^b = \mu^0$, the only way for $\bar \pi_{xt}$ to be supported on the graph of a function is if $T^\D$ is unique, i.e., $T^{\D} = T^{\D,a} = T^{\D,b}$.

A similar argument applies to $\bar \pi_{yt}$, ensuring that $T^\A$ is also unique. If this were not the case, a more efficient coupling could be constructed through monotone rearrangement, contradicting optimality, and thus completes our proof.
\end{proof}

\subsection{Generation to multi-node line graph}

We are now in a position to extend Problem~\ref{prob:problem1} to a line graph with $n$ intermediate vertices. Consider a single path $\mathbf p$ defined as
$
\mathbf p = \{v_0, v_1, \dots, v_{\calT-1}, v_\calT\},
$
and adopt the convention $v_\calT := v_{n+1}$, the corresponding cost function is defined as
\begin{align}\label{eq:cost}
c^{(\mathbf{p})}(t_0, t_1, \dots, t_{\calT-1}, t_\calT) := \sum_{\ell=1}^{\calT}
\frac{w(v_{\ell-1}, v_\ell)}{t^\ell - t^{\ell-1}},
\end{align}
with $w(v_{\ell-1}, v_\ell)$ representing the weight associated with the edge connecting $v_{\ell-1}$ and $v_\ell$.

\begin{customproblem}{$\mathbf 1^\prime$}[\bf \textit{Independent DA on line graph}]\label{prob:prob1-prime}
Given probability distributions $\mu^0(t_0)$ and $\mu^\calT(t_\calT)$, along with suitable flow-rate bounds $r_{1},r_2,\dots,r_{\calT-1} > 0$, determine a probability measure $\pi(t_0, \dots, t_\calT)$ on $\mathcal{M}_+\big([0, t_f]^{n+2}\big)$ that minimizes
\begin{subequations}
\begin{align}
\int_{t_0,t_1,\dots,t_{\calT-1},t_\calT}\!\!\!\!\!\!\!\!\!\!\!\!\!\!\!\!\!\!\!\!\!\!\!\!\!\!\!\!\! c^{(\mathbf{p})}(t_0,t_1,\dots,t_{\calT-1},t_\calT) \,\pi(dt_0, dt_1, \dots, dt_{\calT-1}, dt_\calT),
\end{align}
subject to the departure and arrival constraint:
\begin{align}
&\int_{t_1,\dots,t_{\calT-1},t_\calT}\!\!\!\!\!\!\!\!\!\!\! \pi(dt_0,dt_1, \dots, dt_{\calT-1}, dt_\calT) = \mu^0(t_0), \label{eq:ind-arrival}\\
&\int_{t_0,t_1,\dots,t_{\calT-1}}\!\!\!\!\!\!\!\!\!\!\! \pi(dt_0,dt_1, \dots, dt_{\calT-1}, dt_\calT) = \mu^\calT(t_\calT),\label{eq:ind-depart}
\end{align}
and the flow-rate constraint:
\begin{align}\label{eq:ind-bound}
\int_{t_0,\dots,t_\calT/\{t_\ell\}}\!\!\!\!\!\!\!\!\!\!\!\!\!\!\!\!\!\!\!\!\!\!\!\!\!\! \pi(dt_0,dt_1, \dots, dt_{\calT-1}, dt_\calT) \leq r_\ell\,dt_\ell,\;\ell=1,\dots,\calT-1.
\end{align}
\end{subequations}
\end{customproblem}
\setcounter{problem}{1}

The set of admissible measures for Problem~\ref{prob:prob1-prime} reads
$\Pi:= \big\{ \pi \;\big|\; \pi \mbox{ satisfies } \eqref{eq:ind-arrival}, \eqref{eq:ind-depart}, \mbox{ and } \eqref{eq:ind-bound} \big\}$.
To establish the non-emptiness of $\Pi$ and therefore existence of a feasible $\pi\in\Pi$, consider the ordered sequence of crossing distributions $\sigma^\ell(dt)$ at each intermediate node $v_\ell$. These distributions must satisfy the feasibility condition in Lemma \ref{lemma:feasible-order}. Moreover, each crossing distribution is normalized such that $\int_0^{t_f} \sigma^\ell(dt) = 1$ for all $\ell = 1, \dots, \calT-1$.

To establish the uniqueness of the optimizer $\pi$, we first recall a multi-marginal generalization of the Monge problem (Definition~\ref{def:monge}) in the one-dimensional setting. Specifically, we introduce the generalized Monge condition \cite[Page 24]{rachev1998mass}, which provides a structural criterion for the existence of an optimal transport plan in the form of a Monge map. This condition applies to the MMOT problems where the cost function $c(x_1, x_2, \dots, x_n)$ with $x_i \in \mathbb{R}$ for all $i=1,2,\dots,n$.

\begin{theorem}[\bf \textit{Uniqueness}]\label{thm:unique1-prime}
Problem~\ref{prob:prob1-prime} admits unique minimizer when feasible.
\end{theorem}

\begin{proof}
The uniqueness of the optimizer follows if the cost function $c(t_0,t_1,\dots,t_\calT)$ satisfies the generalized Monge condition in Definition~\ref{def:genmonge}. For any pair $i < j$, the cost function satisfies the Monge condition, ensuring quasi-monotonicity when all other variables are fixed. By Lemma~\ref{lemma:monotone}, the generalized Monge condition follows, resulting in uniqueness.

We next show that the relaxed marginal $\sigma_i$ is unique. Suppose there exist two distinct minimizers, $\pi^a$ and $\pi^b$, inducing different densities $\sigma^a(t_i)$ and $\sigma^b(t_i)$, along with corresponding transport maps $T^{j,a}$ and $T^{j,b}$. The convex combinations
$\bar \sigma = \frac{1}{2} \sigma^a + \frac{1}{2} \sigma^b$ and $\bar \pi = \frac{1}{2} \pi^a + \frac{1}{2} \pi^b
$, then remain feasible. However, since an optimal coupling must be supported on the graph of a function, the convex combination
$$
\bar \pi_{t_i, t_j} = ({\rm Id},  T^{j,a})_{\sharp}\tfrac{1}{2} \sigma^a +
({\rm Id},  T^{j,b})_{\sharp}\tfrac{1}{2} \sigma^b
$$
fails to satisfy this requirement. This indicates the existence of a strictly lower-cost coupling through monotone rearrangement, which contradicts optimality. Consequently, the marginal $\sigma$ is unique for all $i$.
\end{proof}

\section{Coupled DA constraints: formulation and uniqueness}\label{sec:couple}

In this section, we turn to the second regime, where departures and arrivals are no longer specified independently but are tied together through a joint departure--arrival law. On the same line graph as in Section~\ref{sec:ind}, each particle now carries a prescribed pair $(t_0,t_{\calT})$ describing its release and delivery times, so that the DA constraint fixes which departures must match which arrivals. The scheduling problem is to insert crossing times at the intermediate nodes that respect these joint DA specifications while satisfying nodal capacity limits. This leads naturally to an unequal-dimensional OT formulation, in which the two-dimensional DA distribution is mapped to one-dimensional crossing-time profiles along the path. We now formally define the coupled DA constraint as follows.

\begin{definition}[\bf \textit{Coupled DA constraint}]
A coupled DA constraint is a probability measure $\mu^{0,\calT}$ defined on product space $\mathcal{M}_+([0,t_f]^2)$. The measure $\mu^{0,\calT}(dt_0,dt_\calT)$ specifies the joint probability that mass departing at time $t_0$ arrives at time $t_\calT$.
\end{definition}
The coupled DA constraint captures the dependency between departure and arrival times, establishing a structured relationship that governs the temporal dynamics of the transport process, see Fig. \ref{fig:DA}.

\begin{figure}[htb!]
\centering
\includegraphics[width=.7\linewidth]{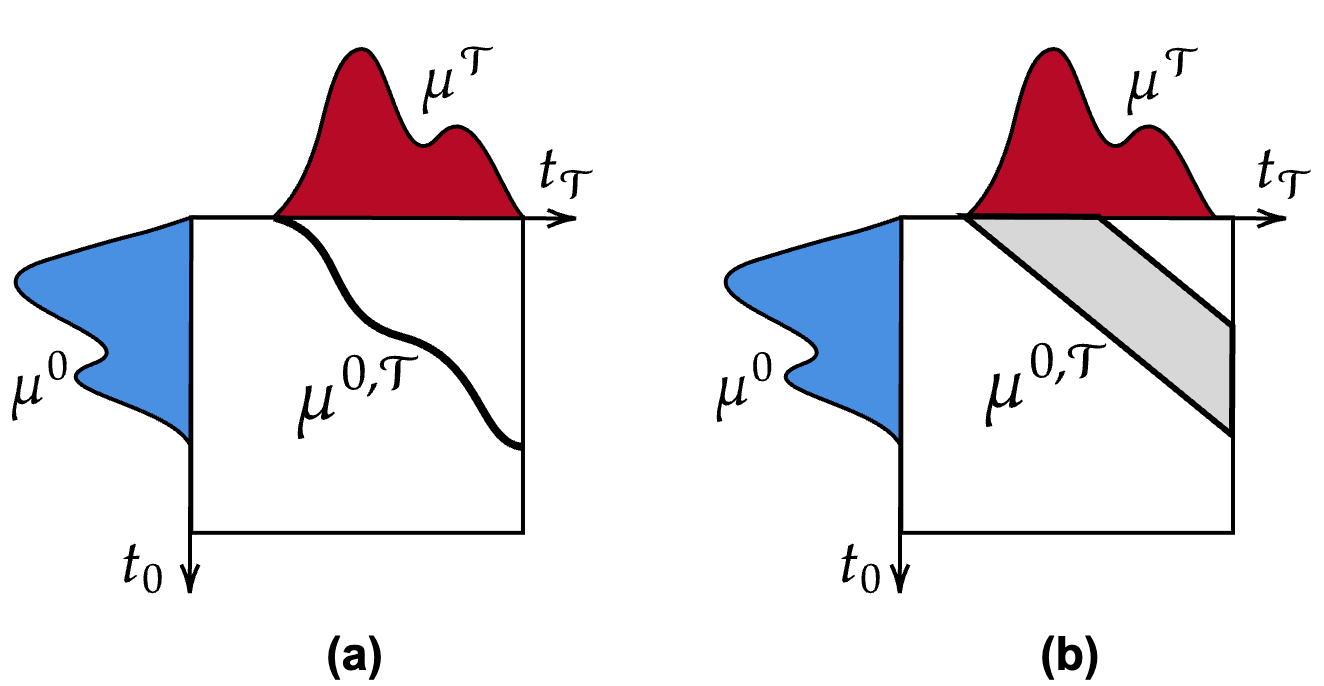}
\caption{Illustration of the coupled DA constraint.
(a) When the joint probability measure $\mu^{0,\calT}$ follows a Monge map, it is supported on the graph of a function, ensuring that each particle is assigned a unique DA time pair. This enforces the assumption that particles departing earlier also arrive earlier.
(b) When $\mu^{0,\calT}$ is not constrained to a function’s graph, the DA relationship is relaxed. Instead of a fixed pair, the arrival time can vary within an interval, introducing flexibility in transport dynamics.}
\label{fig:DA}
\end{figure}

\subsection{Single intermediate node with coupled DA}
Adopting the same assumptions as in Problem~\ref{prob:problem1}, we consider the case of a single toll. Under this setting, the problem can be formulated as follows.
\begin{problem}[\bf \textit{Coupled DA at single node}]\label{prob:prob2}
Given a coupled DA constraint represented by the probability measure $\mu^{0,2}$ and a flow-rate bound $r > 0$, determine a probability measure $\pi(t_0,t_1,t_2)$ on $\mathcal M_+([0, t_f]^3)$ that minimizes
\begin{subequations}
\begin{align}\label{eq:cost-function-3}
\int_{t_0,t_1,t_2} \Bigg( \frac{w(v_0,v_1)}{t_1-t_0} + \frac{w(v_1,v_2)}{t_2-t_1} \Bigg) \pi(dt_0,dt_1,dt_2)
\end{align}
subject to the coupled DA distribution
\begin{align}
\int_{t_1} \pi(dt_0,dt_1,dt_2) = \mu^{0,2}(dt_0,dt_2),
\end{align}
and flow-rate constraint $\int_{t_0,t_\calT} \pi(dt_0,dt_1,dt_2) \leq r dt_1$.
\end{subequations}
\end{problem}

The generalized form of Problem~\ref{prob:prob2} is illustrated in Fig.~\ref{fig:single-path-DA}. Given the joint probability measure governing departure and arrival times, the objective is to determine the optimal crossing times for all particles at each intermediate node. This formulation aligns with the framework of unequal-dimensional OT \cite{chiappori2016multidimensional,chiappori2017multi,chiappori2020multidimensional}, where the two-dimensional coupled DA constraint is mapped to a one-dimensional probability measure characterizing the crossing times at intermediate vertices.
\begin{figure}[htb!]
    \centering
    \includegraphics[width=.7\linewidth]{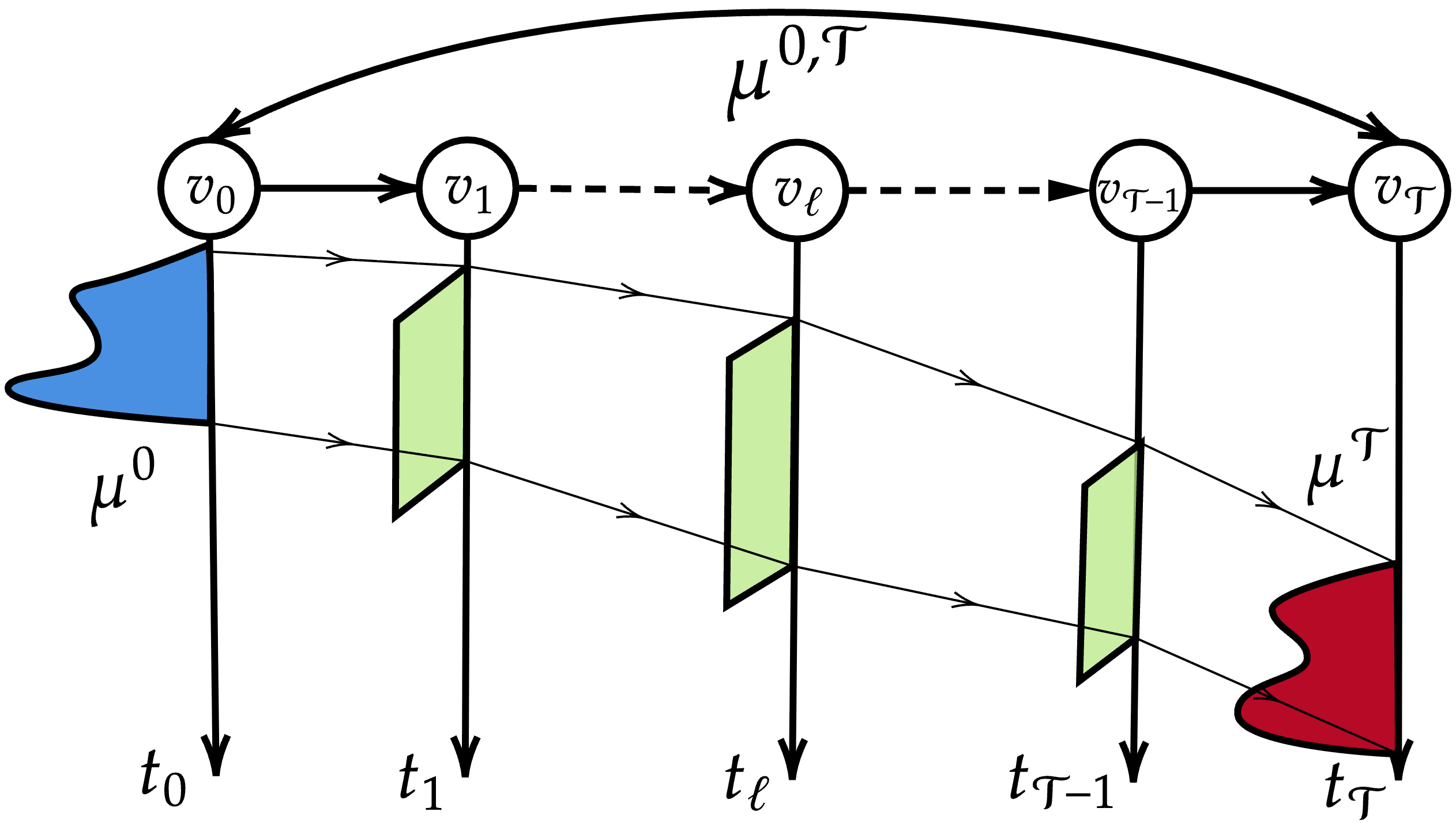}
    \caption{Diagram of the coupled DA constraint over a line graph with source node $v_0$ and sink node $v_\calT$, and connected through intermediate nodes $v_1, \dots, v_{\calT-1}$. The coupled DA $\pi(\mu^0, \mu^\calT)$ encodes the joint probability of departure time $t_0$ and arrival time $t_\calT$. At each intermediate node $v_\ell$, the transport process is characterized by a crossing-time distribution, determining when mass moves from one node to the next.}
    \label{fig:single-path-DA}
\end{figure}
The existence of the problem thus follows directly.
\begin{theorem}[\bf \textit{Existence}]
For sufficiently large flow-rate bound $r$, Problem~\ref{prob:prob2} admits a minimizing solution $\pi$.
\end{theorem}

\begin{proof}
The space of admissible measures $\pi$ in Problem  \ref{prob:prob2} reads
\begin{align*}
\Pi = \Big\{&\pi \in \mathcal{M}_+\big([0, t_f]^3\big) \Big | \int_{t_1} \pi = \mu^{0,2}(dt_0,dt_2),\ \iint_{t_0, t_2} \pi:= \sigma^1(t_1) \leq r , dt_1 \Big\}
\end{align*}
is non-empty. It is compact in the weak topology as it is tight and closed for the narrow convergence. As its cost function is lower semi-continuous, we have the existence of a minimizer. 
\end{proof}

\subsection{Unequal-dimensional OT for uniqueness}
Before further analyzing the properties of Problem~\ref{prob:prob2}, we formally introduce unequal-dimensional OT, where the optimal transportation plan extends to probability measures with different dimensional supports, i.e., $T: \mathbb{R}^n \to \mathbb{R}^m,$ with $m \neq n.$ Previous works ~\cite{chiappori2016multidimensional,chiappori2017multi,chiappori2020multidimensional} have investigated the existence, uniqueness, and smoothness of the optimal map in this setting. To establish the uniqueness of the optimal transport map in the unequal-dimensional setting, we introduce three key concepts that characterize the structure of optimal solutions.

The non-degeneracy condition ensures that the cost function is sufficiently regular to yield meaningful transport maps, preventing degenerate solutions.
\begin{definition}[\bf \textit{Non-degeneracy condition}~\cite{chiappori2020multidimensional}]\label{def:non-deg}
A cost function $c(x, y)$ satisfies the non-degeneracy condition if
$\text{rank}\left(D^2_{xy} c(x,y)\right) = \min\{m, n\}$,
where $n$ and $m$ denote the dimensions of $x$ and $y$, respectively.
\end{definition}

We need conditions that characterize transport plans concentrated on deterministic mappings.
\begin{definition}[\bf \textit{Pure matching}~\cite{chiappori2016multidimensional}]\label{def:pureness}
A stable matching is said to be pure if the transport plan $T$ is concentrated on the graph of a function $T: \mathbb{R}^n \to \mathbb{R}^m$. This implies that almost all agents of type $x$ are matched exclusively with a single agent of type $y = T(x)$.
\end{definition}

A weaker form of the twist condition is also proposed for the unequal-dimensional case, ensuring injectivity from the higher-dimensional space to the lower one without full bijectivity.
\begin{definition}[\bf \textit{$x$-twist condition} \cite{chiappori2020multidimensional}]\label{def:x-twist}
The $x$-twist condition states that, for fixed $y \in \mathbb{R}^m$, the gradient of the cost function w.r.t. $x \in \mathbb{R}^n$ ($n > m$), $\nabla_x c(x,y)$, is injective in $x$, i.e., $y \mapsto \nabla_x c(x,y)$   
is injective for fixed $y$.
\end{definition}

We are now in a position to establish the uniqueness of the transport plan, which maps the coupled DA constraint to a unique crossing time.

\begin{proposition}[\bf \textit{Purity and uniqueness}]\label{prop:pureness}
If the cost function $c(t_0, t_1, t_2)$ satisfies both the non-degeneracy and twist condition, then the minimizer of Problem~\ref{prob:prob2} is unique. Moreover, there exist unique functions $T$ such that
$\pi = ({\rm Id}, T)_{\sharp}\mu^{0,\calT}$,
where ${\rm Id}_{\sharp}\mu^{0,\calT} = \mu^{0,\calT}$ is the identity map, and $T_{\sharp}\mu^{0,\calT} = \sigma$, with $\sigma(t_1)$ an absolutely continuous measure on $[0,t_f]$ with $\sigma(t_1)\leq  r$. 
\end{proposition}

\begin{proof}
We begin by verifying the non-degeneracy condition of the cost functional \eqref{eq:cost-function-3} in Problem \ref{prob:prob2},
and introduce the notation $t_{da}$ for departure--arrival time pair $t_{da} := (t_0, t_2)$. The second-order mixed partial derivative matrix of the cost function is given by
\begin{align*}
\nabla_{t_{da}t_1} c(t_{da}, t_1) = \begin{bmatrix}
\partial^2 c/\partial t_0 \partial t_1\\[0.08in]
\partial^2 c/\partial t_2 \partial t_1
\end{bmatrix} =
2
\begin{bmatrix}
w(v_0,v_1)(t_1 - t_0)^{-3} \\[0.08in]
w(v_1,v_2)(t_2 - t_1)^{-3}
\end{bmatrix}.
\end{align*}
Since this matrix has rank 1, the non-degeneracy condition is satisfied.

Next, we establish that the $x$-twist condition holds. The gradient $\nabla_{t_{da}} c(t_{da}, t_1)$ is
\begin{align*}
\nabla_x c(x, y) = 
\begin{bmatrix}
\partial c/\partial t_0 \\[0.08in]
\partial c/\partial t_2
\end{bmatrix} =
\begin{bmatrix}
-w(v_0,v_1)(t_1 - t_0)^{-2} \\[0.08in]
\phantom{-}w(v_1,v_2)(t_2 - t_1)^{-2}
\end{bmatrix}.
\end{align*}
For any $t_0 < t_1 \neq t'_1 < t_2$, we compute the difference
\begin{align*}
\nabla_{t_{da}} c(t_{da}, t_1) - \nabla_{t_{da}} c(t_{da}, t'_1) =
\begin{bmatrix}
\frac{(t'_1 - t_1)(t'_1 + t_1 - 2t_0)}{(t_1 - t_0)^2(t_1 - t_0)^2} \\[0.08in]
\frac{(t'_1 - t_1)(2t_2 - t'_1 - t_1)}{(t'_1 - t_2)^2(t_1 - t_2)^2}
\end{bmatrix}.
\end{align*}
Since $t_0 <t_1, t'_1 < t_2$, it follows that
$\nabla_{t_{da}} c(t_{da}, t_1) \neq \nabla_{t_{da}} c(t_{da}, t'_1) \quad \text{for every } t_1 \neq t'_1$.
Thus, the $x$-twist condition holds, ensuring that the mapping $t_{da}:=(t_0,t_2) \mapsto \nabla_{t_{da}} c(t_{da}, t_1)$ is injective for each fixed $t_1$. By Definition~\ref{def:pureness}, the transport plan is supported on the graph of a function, implying uniqueness of the stable matching as in \cite[Corollary 4]{chiappori2016multidimensional}. The uniqueness of the relaxed marginal $\sigma(t_1)$ follows from Theorems \ref{thm:unique1} and \ref{thm:unique1-prime}, and we omit repetition.
\end{proof}

\begin{remark}[\bf \textit{Smoothness}]\label{remark:smooth}
When the $x$-twist condition is satisfied, it ensures uniqueness of the optimizer. However, the cost functional in Problem~\ref{prob:prob2} does not generally guarantee smoothness of the optimal coupling. For the coupling to be smooth, a nestedness condition must be met, which depends on the triplet $(c, \mu^{0,\calT}, \mu^1)$. In the case of arbitrary departure and arrival measures, one must verify whether the index or pseudo-index form for the cost is satisfied. We refer to \cite{dong2025task,chiappori2020multidimensional} for further discussion.
\end{remark}

\subsection{Multi-node line graph with coupled DA}

We now pass from the single-node setting of Problem~\ref{prob:prob2} to a line graph with $n$ intermediate vertices. Consider a path $\mathbf{p} = \{v_0,v_1,\dots,v_{\calT-1},v_{\calT}\}$, with $v_{\calT}:=v_{n+1}$. The coupled DA law $\mu^{0,\calT}$ still prescribes which departure–arrival pairs $(t_0,t_{\calT})$ must be realized, while the crossing times at the intermediate nodes are now chosen so as to respect this joint DA specification and, at the same time, satisfy the capacity constraints at each node along the path.

\begin{customproblem}{$\mathbf 2^\prime$}[\bf \textit{Coupled DA on a line graph}]\label{prob:prob2-prime}
Given a coupled departure–arrival law $\mu^{0,\calT}$ on $[0,t_f]^2$ and flow-rate bounds $r_1,\dots,r_n > 0$ at the intermediate nodes, determine a probability measure $\pi(t_0,\dots,t_{\calT-1},t_\calT)$ on $\mathcal M_+([0,t_f]^{n+2})$ that minimizes
\begin{subequations}
\begin{align}\label{eq:obj2-prime}
\int_{t_0,\dots,t_n,t_\calT}\!\!\!\!\!\!\!\!\!\!\!\!\!\!\!\!\!\!
c^{(\mathbf{p})}(t_0,\dots,t_\calT)\,
\pi(dt_0,\dots,dt_\calT),
\end{align}
where $c^{(\mathbf{p})}$ is the path cost \eqref{eq:cost}, for $t_0 < t_1 < \dots < t_{\calT-1} < t_\calT$, and subject to coupled DA constraint
\begin{align}\label{eq:couple-da}
\int_{t_1,\dots,t_{\calT-1}} \!\!\!\!\!\!\!\!\!\!\!\!\! \pi(dt_0,\dots,dt_\calT)
= \mu^{0,\calT}(dt_0,dt_\calT),
\end{align}
and the nodal flow-rate constraints
\begin{align}\label{eq:couple-bound}
\int_{t_0,\dots,t_\calT/\{t_\ell\}} \!\!\!\!\!\!\!\!\!\!\!\!\!\!\!\!\!
\pi(dt_0,\dots,dt_\calT)
:= \sigma^\ell(dt_\ell) \le r_\ell\,dt_\ell,
\end{align}
\end{subequations}
for $\ell = 1,\dots,n.$
\end{customproblem}
\setcounter{problem}{2}

The admissible set for Problem~\ref{prob:prob2-prime} reads 
$\Pi := \big\{\pi \; \big|\; \pi \mbox{ satisfies } \eqref{eq:couple-da} \mbox{ and } \eqref{eq:couple-bound} \big\}$.
As in the independent DA case, one can construct intermediate crossing laws
$\sigma^1,\dots,\sigma^n$ along the path that satisfy the feasibility condition in Lemma~\ref{lemma:feasible-order} and are normalized so that
$\int_0^{t_f} \sigma^\ell(dt) = 1$ for all $\ell$. The gluing lemma then provides at least one coupling $\pi\in\Pi$ with these marginals.

We now turn to uniqueness and structure. Introduce the DA variable $t_{da} := (t_0,t_\calT) \in [0,t_f]^2$, and gather the crossing times in $\mathbf t := (t_1,\dots,t_{\calT-1}) \in [0,t_f]^n$. Viewed as a function of $(t_{da},\mathbf t)$, the cost $c^{(\mathbf{p})}$ fits the unequal-dimensional OT setting of Section~\ref{sec:couple}, where $t_{da}$ and $\mathbf t$ play the roles of the higher- and lower-dimensional variables.

\begin{theorem}[\bf \textit{Purity and uniqueness}]\label{thm:unique-coupled-line}
Assume that the path cost $c^{(\mathbf{p})}(t_0,\dots,t_{\calT-1},t_\calT)$ in \eqref{eq:cost} satisfies the non-degeneracy condition and the $x$-twist condition with respect to $t_{da}=(t_0,t_\calT)$. Then, whenever Problem~\ref{prob:prob2-prime} is feasible, it admits a unique minimizer $\pi$. Moreover, there exists a unique measurable map $T^{(\mathbf p)} : [0,t_f]^2 \to [0,t_f]^n,$ such that
\[
\pi = \big({\rm Id},T^{(\mathbf p)}\big)_{\sharp}\mu^{0,\calT},
\]
and the induced crossing-time marginals $\sigma^\ell$ are absolutely continuous and satisfy $\sigma^\ell(t_\ell)\le r_\ell$ for all $\ell=1,\dots,n$.
\end{theorem}

\begin{proof}
Introduce $t_{da}=(t_0,t_\calT)$ and $\mathbf t=(t_1,\dots,t_{\calT-1})$. The path cost \eqref{eq:cost} depends on $t_0$ only through the first segment and on $t_\calT$ only through the last. Hence the mixed second derivatives with respect to $t_{da}$ and $\mathbf t$ vanish everywhere except in the entries $(t_0,t_1)$ and $(t_{\calT-1},t_\calT)$, where they are strictly positive on the feasible set $t_0 < t_1 < \dots < t_{\calT-1} < t_\calT$. The mixed derivative matrix $\nabla_{t_{da}\mathbf t}c^{(\mathbf p)}$ therefore has rank $\min\{2,n\}$, so the non-degeneracy condition of Definition~\ref{def:non-deg} holds.

For the $x$-twist condition, again, only the first and last segments matter. The gradient with respect to $t_{da}$ is
\begin{align*}
\nabla_{t_{da}} c^{(\mathbf p)}(t_{da},\mathbf t)
=
\begin{bmatrix}
-w(v_0,v_1)(t_1-t_0)^{-2}\\[0.08in]
w(v_n,v_\calT)(t_\calT-t_{\calT-1})^{-2}
\end{bmatrix}.
\end{align*}
For fixed $\mathbf t$, the first component is strictly decreasing in $t_0$ and the second is strictly increasing in $t_\calT$. Thus, for any $(t_0,t_\calT)\neq(t_0',t_\calT')$ with the same $\mathbf t$,
\[
\nabla_{t_{da}} c^{(\mathbf p)}(t_0,t_\calT,\mathbf t)
\neq
\nabla_{t_{da}} c^{(\mathbf p)}(t_0',t_\calT',\mathbf t),
\]
so $t_{da}\mapsto \nabla_{t_{da}} c^{(\mathbf p)}(t_{da},\mathbf t)$ is injective and the $x$-twist condition (Definition~\ref{def:x-twist}) is satisfied.

Non-degeneracy and $x$-twist place us in the unequal-dimensional OT setting of
\cite{chiappori2016multidimensional,chiappori2017multi,chiappori2020multidimensional}. In particular, any stable matching between $t_{da}$ and $\mathbf t$ is pure, so any minimizer of Problem~\ref{prob:prob2-prime} can be written as
\[
\pi = \big({\rm Id},T^{(\mathbf p)}\big)_{\sharp}\mu^{0,\calT},
\]
with unique measurable map $T^{(\mathbf p)}$. The capacity constraints act only on the relaxed marginals $\displaystyle \sigma^\ell$, which are the push-forwards of $\mu^{0,\calT}$ under the components of $T^{(\mathbf p)}$. If there were two distinct minimizers with different maps (and hence different collections of crossing laws), their convex combination would still satisfy all marginal and capacity constraints but would no longer be supported on the graph of a single function of $t_{da}$. As in the proof of Theorem~\ref{thm:unique1-prime}, a monotone rearrangement argument would then yield a strictly cheaper coupling, contradicting optimality. Thus $T^{(\mathbf p)}$, the induced crossing marginals $\sigma^\ell$, and the minimizer $\pi$ are all unique.
\end{proof}

\section{Multi-path generalization and optimization methodology}\label{sec:general}

We recall that $\calV^-\subseteq\calV$ and $\calV^+\subseteq\calV$ denote the sets of source and sink nodes, and that $\calP(\calV^-,\calV^+)$ is a collection of admissible paths from sources to sinks. 
For a given node $v\in\calV$, we write $\calP(v)$ for the subsets of paths in $\calP(\calV^-,\calV^+)$ that contain the node $v$.
The general form of Problem~\ref{prob:problem1} is then expressed in terms of a path-wise transportation plan $\pi^{(\bp)} \in \mathcal{M}_{+}\big([0,t_f]^{n_\bp}\big)$: for each admissible path (line graph) $\bp \in \calP(\calV^-,\calV^+)$, the joint coupling measure $\pi^{(\bp)}$ describes the evolution of mass across the $n_\bp$ vertices along $\bp$, and the objective is to minimize the total transportation cost over all commodities and all paths. 
The problem is thus formulated as follows.

\begin{problem}[\bf \textit{Path-wise DA-constrained transport}]\label{prob:problem3}
Given probability measures $\mu^0(t_0,v^-_i)$ for $v^-_i\in \calV^-$ and $\mu^\calT(t_\calT,v^+_j)$ for $v^+_j\in \calV^+$, and flow-rate bounds $r(t_\ell,v_\ell)>0$ at every intermediate node $v_\ell$. We want to determine $\pi^{(\bp)}\in\calM_{+}([0,t_f]^{n_\bp})$ for all $\bp\in\calP(\calV^-,\calV^+)$ that minimizes
\begin{subequations}
\begin{align*}
\sum_{\bp\in\calP(\calV^-,\calV^+)} \!\!\!\!\!\!\!\! \big\langle c^{(\bp)},\pi^{(\bp)}\big\rangle
:=\!\!\!\!\!\!\!\!
\sum_{\bp\in\calP(\calV^-,\calV^+)} \int c^{(\bp)}\,d\pi^{(\bp)},
\end{align*}
subject to the departure and arrival marginals
\begin{align*}
\sum_{\bp\in\calP(v^-_i)} P_{v^-_i}\big(\pi^{(\bp)}\big)=\mu^0(t_0,v^-_i), \quad v^-_i\in\calV^-\\
\sum_{\bp\in\calP(v^+_j)} P_{v^+_j}\big(\pi^{(\bp)}\big)=\mu^\calT(t_\calT,v^+_j) \quad v^+_j\in\calV^+,
\end{align*}%
and the flow rate capacity constraints
\[
\sum_{\bp\in\calP(v_\ell)} P_{v_\ell}\big(\pi^{(\bp)}\big)\;\le\; r(t_\ell,v_\ell)\,dt_\ell, \quad \mbox{ for } v_\ell\in\calV,
\]
where $P_{v}:\mathcal{M}_{+}\big([0,t_f]^{n_\bp}\big)\to \mathcal{M}_{+}([0,t_f])$ denotes the projection onto the marginal associated with node $v$. 
\end{subequations}

\end{problem}

A motivation for this formulation is that given an initial and final destination, there are typically a limited number of effective paths. However, for a single path with $n$ nodes, the size of the corresponding transportation map $\pi^{(\bp)}$ grows exponentially as the number of temporal marginals at each node increases. To address this challenge, we develop methods based on multimarginal optimal transport. In Section~\ref{sec:sinkhorn}, we review multi-marginal entropic regularization and its Sinkhorn implementation. In Section~\ref{sec:path-sinkhorn}, we adapt this framework to a path-wise Sinkhorn algorithm tailored to DA-constrained transport on networks.

\subsection{Multi-marginal Sinkhorn algorithm, recap}\label{sec:sinkhorn}

For the MMOT problem, the discretized formulation with objective functional in \eqref{eq:multi-marginal-obj} and admissible set \eqref{eq:multi-marginal-set} yields a large-scale linear program. In the multi-marginal setting, the computational cost typically grows exponentially with the number of marginals, so direct solvers quickly become impractical. The situation improves substantially when the cost has a graphical structure, for instance, when it decomposes into pairwise terms,
$$
c(x_0,\dots,x_\calT)
= c(x_0,x_1)+c(x_1,x_2)+\dotsb+c(x_{\calT-1},x_\calT).
$$

In such cases, entropic regularization combined with Sinkhorn-type scaling has become a standard tool for large-scale MMOT. Recent work shows that tensor factorization can further accelerate these methods by sharing structure across marginals and removing redundant projections within each iteration, see e.g. \cite{elvander2020multi,haasler2021multimarginal,haasler2024scalable} for details and applications.

By adding an entropic penalty to the original problem, the regularized version can be solved using an alternating minimization scheme \cite{peyre2019computational}. The key idea of entropic regularized OT, is to introduce the entropic penalty term of $\pi \in \mathbb{R}_+^{n_0\times\cdots\times n_\calT}$ defined as 
\begin{align}\label{eq:entropic-term}
    D(\pi):= - \sum_{i_0,\dots,i_\calT} \pi_{i_0,\dots,i_\calT}\Big(\log \pi_{i_0,\dots,i_\calT}-1\Big),
\end{align}
with the convention $0 \log(0) = 0$. With a regularization parameter $\epsilon$ to penalize the entropic term, the entropic regularization problem can be written as
\begin{align}\label{prob:entropy}
\pi_\epsilon^\star:=   \argmin_{\pi \in \Pi(\mu^0,\mu^1,\dots,\mu^\calT)} \ \langle c,\pi \rangle - \epsilon D(\pi).
\end{align}
The unique solution $\pi^\star$ of \eqref{prob:entropy} converges to the optimal solution with maximal entropy within the set of all optimal solutions $\pi\in\Pi$, see \cite[Proposition 4.1]{peyre2019computational}. 

To obtain an iterative updating scheme for finding the optimizer of $\pi^\star_\epsilon$, we derive the Lagrangian 
\begin{align*}
\calL(\pi,\lambda^{0,\dots,\calT}) =
\langle c,\pi \rangle - \epsilon D(\pi)
+\sum_{\ell=0}^{\calT} \lambda^{\ell} \left(P_\ell(\pi) -\mu^\ell\right),
\end{align*}
where, for each $\ell$, the marginal $\mu^\ell \in \mathbb{R}_+^{n_\ell}$ and Lagrange dual variable $\lambda^\ell \in \mathbb{R}^{n_\ell}$ for the discretized marginal constraint $P_\ell(\pi)\in\mathbb{R}_+^{n_\ell}$. Then, by taking the first-order optimality condition, we have 
\begin{align*}
&\partial\calL(\pi,\lambda^0,\lambda^1,\dots,\lambda^\calT)/\partial \pi = 0,\\
&\partial\calL(\pi,\lambda^0,\lambda^1,\dots,\lambda^\calT)/\partial \lambda^\ell = 0,
\end{align*}
which gives that the optimizer $\pi$ has a closed-form expression 
$$\pi_{i_0,\dots,i_\calT} = \exp(-c_{i_0,\dots,i_\calT}/\epsilon)\cdot \prod_{\ell=0}^{\calT}\exp(-\lambda_{i_\ell}^{\ell}/\epsilon).$$

By introducing the notation $\bK:=\exp(-c/\epsilon)$ and $\bu_\ell:= \exp(-\lambda_{\ell}/\epsilon)$, the optimizer can be written in the more compact form\footnote{In the language of Schr\"odinger bridge problem, where only the entropic term is considered, taking the limit $\epsilon \to +\infty$ reduces the problem to maximum likelihood estimation. In this setting, the reference measure $\bK$ acts as the prior, while the solution $\pi$ corresponds to the posterior.}
\begin{align}\label{eq:close-form}
\pi = \bK \odot \left(\bu_0 \otimes \bu_1 \dots \otimes \bu_\calT \right)
\end{align}
with $\odot$ the element-wise product and $\otimes$ as the outer product. The dual variables $\lambda_0,\lambda_1,\dots,\lambda_\calT$ have to satisfy the marginal constraints, which follow after we derive the dual by substituting \eqref{eq:close-form} into the Lagrangian
\begin{align*}
\argmax_{\lambda^0,\lambda^1,\dots,\lambda^\calT} - \epsilon \bK \odot \left(\bu_0 \otimes \bu_1 \dots \otimes \bu_\calT \right) + \sum_{i=0}^{\ell} \langle \bu_\ell, \mu^\ell \rangle.
\end{align*}
The updating strategy thus initializes $\bu_0=\ones_{n_0}\in\mathbb{R}_+^{n_0}$, 
and then for $\ell=0:\calT$, we iteratively compute
\begin{equation}\label{eq:time-update}
\bu_\ell \longleftarrow \mu^\ell\oslash P_{\ell}(\pi)
\end{equation}
until convergence, where $\oslash$ denotes element-wise division. The above updating strategy is known as the Sinkhorn algorithm \cite{dong2025data,peyre2019computational,sinkhorn1964relationship}, which also can be viewed as matrix scaling, by iteratively modifying $\bu_0,\dots,\bu_\calT$ to satisfy the marginal constraint. In addition, for the relaxed marginal, i.e., $P_\ell(\pi) \leq \mu^\ell$, the updates are 
$$
\bu_\ell \longleftarrow \max \Big\{ \mu^\ell\oslash P_{\ell}(\pi),\; \mathbf 1\Big\}.
$$
This approach is analogous to the standard Sinkhorn method, and thus inherits a linear convergence rate \cite{haasler2021multi}.

\subsection{Path-wise Sinkhorn algorithm}\label{sec:path-sinkhorn}
We next derive a Sinkhorn-type scaling scheme for the path-wise formulation in Problem~\ref{prob:problem3}. Each admissible path $\bp$ carries its own coupling $\pi^{(\bp)}$ over the crossing times of the nodes it visits, while the source and sink marginals and interior-node capacities couple \emph{across} paths through aggregated projections. Entropic regularization yields a multiplicative structure and simple alternating matrix rescaling, in the spirit of the multi-marginal case, with the difference that each interior node shares one multiplier wherever it appears.

In practice it is often unnecessary to work with the full all-path formulation. Instead, one may restrict attention to a prescribed family of admissible paths and enlarge it only when needed, in the spirit of column generation for large-scale linear programs \cite{friesecke2022genetic,nemhauser2012column}. In network flow models with path variables, such enrichment is typically guided by a shortest-path-type pricing subproblem under dual weights \cite{engineer2008shortest,moradi2015bi}.

For a discrete path plan $\pi^{(\bp)} \in \mathbb{R}_+^{n_0 \times \cdots \times n_{n_\bp} \times n_{\calT}}$, we consider the path-wise entropy $D(\pi^{(\bp)})$ defined in \eqref{eq:entropic-term}, and add it to the cost functional in Problem~\ref{prob:problem3}, yielding its entropically regularized formulation. By introducing dual variables $\alpha_{v^-}(t_0)$, $\beta_{v^+}(t_{\calT})$, and $\gamma_{v}(t_\ell)\ge 0$ corresponding to the source node, sink node, and flow rate constraints, respectively, the Lagrangian can be written as


\begin{align*}
\mathcal L(\{\pi^{(\bp)}\},&\alpha,\beta,\gamma)
=
\sum_{\bp}
\Big\langle c^{(\bp)} + \varepsilon\big(\log \pi^{(\bp)} - 1\big),\,\pi^{(\bp)} \Big\rangle 
\hspace{-0.4pt}-\sum_{v^-\in \calV^{-}} \Big\langle \alpha_{v^-},\!\!\!\!
\sum_{\bp\in\calP(v^-)} P_{v^-}\big(\pi^{(\bp)}\big) - \mu^0(\cdot,v^-) \Big\rangle \nonumber\\
&
-\sum_{v^+\in \calV^+} \Big\langle \beta_{v^+},\!\!\!\!
\sum_{\bp\in\calP(v^+)} P_{v^+}\big(\pi^{(\bp)}\big) - \mu^{\calT}(\cdot,v^+) \Big\rangle 
-\sum_{v\in \calV} \Big\langle \gamma_v,\!\!\!\!
\sum_{\bp\in\calP(v)} P_{v}\big(\pi^{(\bp)}\big) - r(\cdot,v) \Big\rangle ,
\end{align*}
where $d\mathbf t$ denotes integration over $d\mathbf t:=(t_0,\{t_\ell:\ v_\ell\in\bp\},t_{\calT})$ and $P_0$, $P_{\calT}$, $P_\ell$ are the marginal operators extracting the $t_0$, $t_{\calT}$, $t_\ell$ marginals of $\pi^{(\bp)}$. The First-order optimality condition of $\mathcal L$ with respect to $\pi^{(\bp)}$ gives
\begin{align*}
0
= c^{(\bp)}+\varepsilon \log \pi^{(\bp)}
\!-\!\alpha_{v^-}(t_0)
\!-\!\!\sum_{v_\ell\in\bp}\gamma_\ell(t_\ell)
\!-\!\beta_{v^+}(t_{\calT}) .
\end{align*}
Hence, the path-wise optimizer admits a closed-form solution in the form
\begin{align*}
\pi^{(\bp)}(t_0,\{t_\ell\},t_{\calT}) =\exp\!\left(
\frac{\alpha_{v^-(\bp)}(t_0)
+\sum_{v_\ell\in\bp}\gamma_\ell(t_\ell)
+\beta_{v^+(\bp)}(t_{\calT})
- c^{(\bp)}}{\varepsilon}
\right).
\end{align*}

Next, define the Gibbs kernel and multiplicative scalings
$\bK^{(\bp)}(t_0,\{t_\ell\},t_{\calT}):=\exp\!\big(-c^{(\bp)}(t_0,\{t_\ell\},t_{\calT})/\varepsilon\big)$, together with
\begin{align*}
\bu_{v^-}\!\!:=\exp\!\big(\frac{\alpha_{v^-}}{\varepsilon}\big),  \bw_v:=\exp\!\big(\frac{\gamma_v}{\varepsilon}\big), \mbox{ and } \bv_{v^+}:=\exp\!\big(\frac{\beta_{v^+}}{\varepsilon}\big).
\end{align*}
Then the path-wise plan may be factorized as (cf.
\cite{elvander2020multi})
\begin{equation}\label{eq:closed-form}
\pi^{(\bp)}=\bK^{(\bp)}\;\bu_{v^-}(t_0)\;\Big(\prod_{v\in\bp} \bw_{v}(t)\Big)\;\bv_{v^+}(t_{\calT}).
\end{equation}

Fix a path $\bp$, and define the path-wise flux profiles as partial integrals of the Gibbs kernel with all other factors, holding one variable and integrating over the rest, so that
\begin{align}\label{eq:flux-profile}
&A^{(\bp)}_{v^-}(t_0):=\int_{\{t_\ell\},t_\calT} \!\!\!\!\!\!\!\!\!\! \bK^{(\bp)}\Big(\prod_{v_\ell\in\bp} \bw_{v_\ell}(t_\ell)\Big) \bv_{v^+}(t_{\calT}),\nonumber \\
&A^{(\bp)}_{v^+}(t_{\calT}):=\int_{t_0,\{t_\ell\}} \!\!\!\!\!\!\!\!\!\! \bK^{(\bp)}\bu_{v^-}(t_0)\Big(\prod_{v_\ell\in\bp} \bw_{v_\ell}(t_\ell)\Big),\\
&A^{(\bp)}_{v}(t_\ell):=\int_{t_0,\{t_{\ell'\neq \ell}\},t_{\calT}} \!\!\!\!\!\!\!\!\!\!\!\!\!\!\!\!\!\!\!\!\! \bK^{(\bp)}\bu_{v^-}(t_0)\bv_{v^+}(t_{\calT})
\!\!\!\!\prod_{v_{\ell'}\in\bp,\ \ell'\neq \ell}\!\!\!\!\!\!\! \bw_{v_{\ell'}}(t_{\ell'}).\nonumber
\end{align}
The model marginals are obtained by summing these flux profiles over paths incident to each node and multiplying by the local scaling as 
\begin{align}\label{eq:aggregate}
m_{v^-}(t_0)
&:= \sum_{\bp\in\calP(v^-)} \bu_{v^-}(t_0)\,A^{(\bp)}_{v^-}(t_0),\nonumber\\
m_{v^+}(t_{\calT})
&:= \sum_{\bp\in\calP(v^+)} \bv_{v^+}(t_{\calT})\,A^{(\bp)}_{v^+}(t_{\calT}),\\
m_{v}(t)
&:= \sum_{\bp\in\calP(v)} \bw_{v}(t)\,A^{(\bp)}_{v}(t),
\qquad v\in\calV.\nonumber
\end{align}


The corresponding path-wise Sinkhorn scheme can be summarized as in Algorithm \ref{alg:sinkhorn_graph}.
\begin{algorithm}[htb!]
\caption{Path-wise Sinkhorn with nodal capacities}\label{alg:sinkhorn_graph}
\begin{algorithmic}[1]
\STATE \textbf{Initialize} $\bu_{v^-}(t_0)\!=\!\mathbf 1$, $\bv_{v^+}(t_{\mathcal T})\!=\!\mathbf 1$, $\bw_v(t)\!=\!\mathbf 1$
\WHILE{ not converged }
\STATE Compute flux profile along each path $\bp$ as in \eqref{eq:flux-profile}.
\STATE Aggregate marginals at nodes \eqref{eq:aggregate}, and compute independent DA projections
{\setlength{\abovedisplayskip}{5pt}%
 \setlength{\belowdisplayskip}{5pt}%
$$
\bu_{v^-} \leftarrow \frac{\mu^0(\cdot,v^-)}{m_{v^-}} 
\ \mbox{ and } \ 
\bv_{v^+} \leftarrow \frac{\mu^{\mathcal T}(\cdot,v^+)}{m_{v^+}}.
$$}
\STATE Compute capacity projection at node $v$, i.e.,
{\setlength{\abovedisplayskip}{5pt}%
 \setlength{\belowdisplayskip}{5pt}%
$$
\bw_v \leftarrow \min\!\Big\{r_v \oslash m_v, \ \mathbf 1 \Big\}\qquad \text{for each } v \in \calV.
$$}
\ENDWHILE
\STATE Compute plan $\pi^{(\bp)}$ for path $\bp$ as in \eqref{eq:closed-form}.

\end{algorithmic}
\end{algorithm}

\begin{theorem}[\bf \textit{Linear convergence}]\label{thm:convergence}
Algorithm~\ref{alg:sinkhorn_graph} converges to the unique optimal plan at a linear rate. 
\end{theorem}

\begin{proof}
The entropic formulation is strictly convex in $\{\pi^{(\bp)}\}$, hence the optimizer is unique and strong duality holds. Each multiplicative update in Algorithm \ref{alg:sinkhorn_graph} enforces one set of marginals or capacity constraints in closed-form, which is the exact maximizer of the dual over that coordinate block. The algorithm is therefore exact block coordinate ascent \cite{luo1992convergence} on a smooth concave function. Standard results for such schemes give monotone ascent and a global linear rate \cite[Theorem 2.1]{luo1992convergence}.
\end{proof}

\begin{remark}[\bf \textit{Stopping criteria on Algorithm \ref{alg:sinkhorn_graph}}]
For the stopping condition of Algorithm \ref{alg:sinkhorn_graph}, we monitor marginal violations and capacity violations, i.e.,
$E_0=\sum_i\|m^0_i-\mu^0(\cdot,v^-_i)\|_1$,
$E_{\calT}=\sum_j\|m^{\calT}_j-\mu^{\calT}(\cdot,v^+_j)\|_1$, and
$V=\sum_\ell\|(m_\ell-r_\ell)_+\|_1$. 
Algorithm \ref{alg:sinkhorn_graph} terminates when $E_0, E_{\calT}, V\le\tau$, with $\tau$ being the precision tolerance. 
\end{remark}

\begin{remark}[\bf \textit{Coupled DA}]
For the coupled DA case, where the boundary marginal is prescribed jointly in $(t_0,t_{\mathcal T})$, the update in Algorithm~\ref{alg:sinkhorn_graph} is modified as follows. Define 
$$m^{0,\mathcal T}(t_0,t_{\mathcal T})
:= \sum_{\bp}\int \pi^{(\bp)}(t_0,\{t_\ell\},t_{\mathcal T})\,d\{t_\ell\},$$ and set the boundary ratio $\Lambda(t_0,t_{\mathcal T}) := \mu^{0,\mathcal T}(t_0,t_{\mathcal T}) \oslash m^{0,\mathcal T}(t_0,t_{\mathcal T})$. The direct projection step is then applied path-wise via 
$$\pi^{(\bp)} \;\leftarrow\; \Lambda \odot \pi^{(\bp)}$$
for all $\bp$, while the interior capacity projections remain unchanged. Similar computations of projections in multimarginal problems with joint marginals constraints were also considered in  \cite{haasler2020optimal,haasler2024scalable}.

\end{remark}

\section{Numerical examples}\label{sec:num}
We now illustrate the two DA regimes and the path-wise formulation on simple test cases. The focus is on how schedules adapt to DA specifications and nodal capacity limits, rather than on numerical performance. Section~\ref{sec:indep-num} considers independent DA constraints on a line graph. Section~\ref{sec:coupled-num} turns to coupled DA constraints on the same geometry. Section~\ref{sec:path-network-num} then studies transport on a small path-specified network, while Section~\ref{sec:convergence-num} reports the convergence of the path-wise Sinkhorn solver.

\subsection{Independent DA constraints on line graph}\label{sec:indep-num}

First, we consider the independent DA setting on a single path. The boundary time marginals at the source $v_0$ and sink $v_{\mathcal T}$ are fixed Gaussian mixtures. The crossing-time marginal at the interior node $v_1$ is discretized on a uniform grid and is subject to a per-time-slice capacity. As in Fig. \ref{fig:independent_DA}, we use the discretization size $n_t=100$ and a capacity $r=2/100$. The optimization enforces the three marginal constraints together with the capacity bound.

\begin{figure}[htb!]
\centering
\includegraphics[width=.75\linewidth]{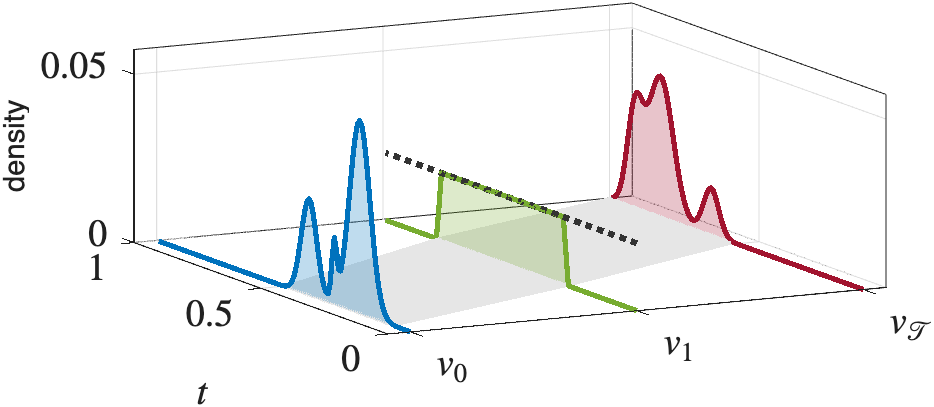}
\caption{\emph{Independent DA on line graph with single intermediate node.} The three one-dimensional time profiles are placed on $v_0$, $v_1$, and $v_{\mathcal T}$. Blue shows the departure density at $v_0$. Green shows the crossing-time density at $v_1$. Red shows the arrival density at $v_{\mathcal T}$. The dotted line at $v_1$ marks the per-time-slice capacity $c=2$. Light gray guides indicate the aggregated couplings along $(v_0,v_1)$ and $(v_1,v_{\mathcal T})$. The plot highlights scheduling: mass is metered through $v_1$ to respect capacity, and node-wise marginals satisfy conservation.}
\label{fig:independent_DA}
\end{figure}

The panel in Fig.~\ref{fig:independent_DA_uni} depicts the endpoint coupling between $v_0$ and $v_{\mathcal T}$ as a smooth surface over $(t_0,t_{\mathcal T})$. The blue and red boundary ridges reproduce the prescribed departure and arrival time marginals at $v_0$ and $v_{\mathcal T}$. The green curve along the far edge encodes the crossing-time marginal at the interior node $v_1$. Dominant triplets $(t_0,t_1,t_{\mathcal T})$ are shown as colored markers, with color indicating the crossing time $t_1$. In one dimension with a flux limit, the separable-in-time objective satisfies a generalized Monge condition, so optimal couplings through $t_1$ concentrate on the graph of a monotone map. The support therefore collapses to a single smooth strand rather than a sheet, and slices align with a function graph--consistent with monotone rearrangement and uniqueness of the optimizer.

\begin{figure}[htb!]
\centering
\includegraphics[width=.75\linewidth]{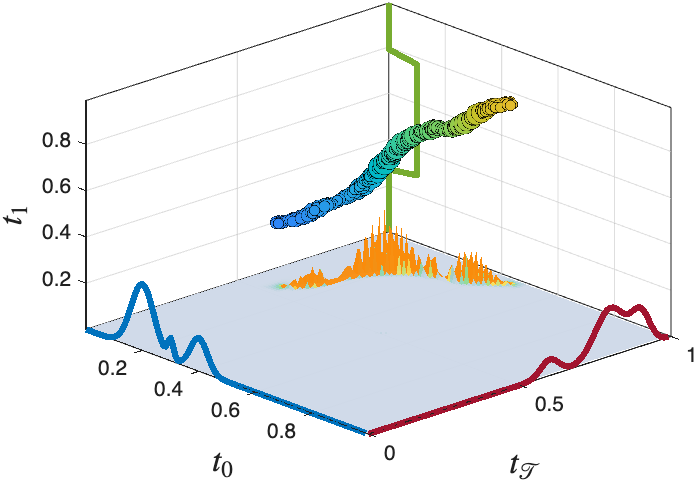}
\caption{\emph{Endpoint coupling and triplet structure for case in Fig.~\ref{fig:independent_DA}.} The surface shows coupling over $(t_0,t_{\mathcal T})$ with blue and red boundary ridges equal to the source and sink marginals. The green curve along the far edge plots the crossing-time marginal at $v_1$. Colored markers indicate dominant triplets $(t_0,t_1,t_{\mathcal T})$ (color = $t_1$). The support collapses to a single smooth strand, in agreement with the monotone-map structure and uniqueness in one-dimensional flux-limited setting.}
\label{fig:independent_DA_uni}
\end{figure}

Next, we extend the independent DA setting to a multistage path with five interior crossings. The joint coupling is supported on tuples $(t_0,t_1,t_2,t_3,t_4,t_5,t_6)$ with $t_0<t_1<\cdots<t_5<t_6$. The departure and arrival time marginals at $v_0$ and $v_{\mathcal T}$ are fixed, while each interior node $v_k$ carries a time-varying capacity profile. Fig.~\ref{fig:independent_DA_line} places the seven one-dimensional time marginals on the rails $v_0,v_1,\dots,v_5,v_{\mathcal T}$ with blue at $v_0$, green at $v_1-v_5$ with dotted capacity traces, and red at $v_{\mathcal T}$. The plot highlights scheduling along the chain: mass is paced across stages so conservation holds at each rail, and every interior rail respects its time-varying capacity, while the boundary marginals are preserved.

\begin{figure}[htb!]
\centering
\includegraphics[width=0.8\linewidth]{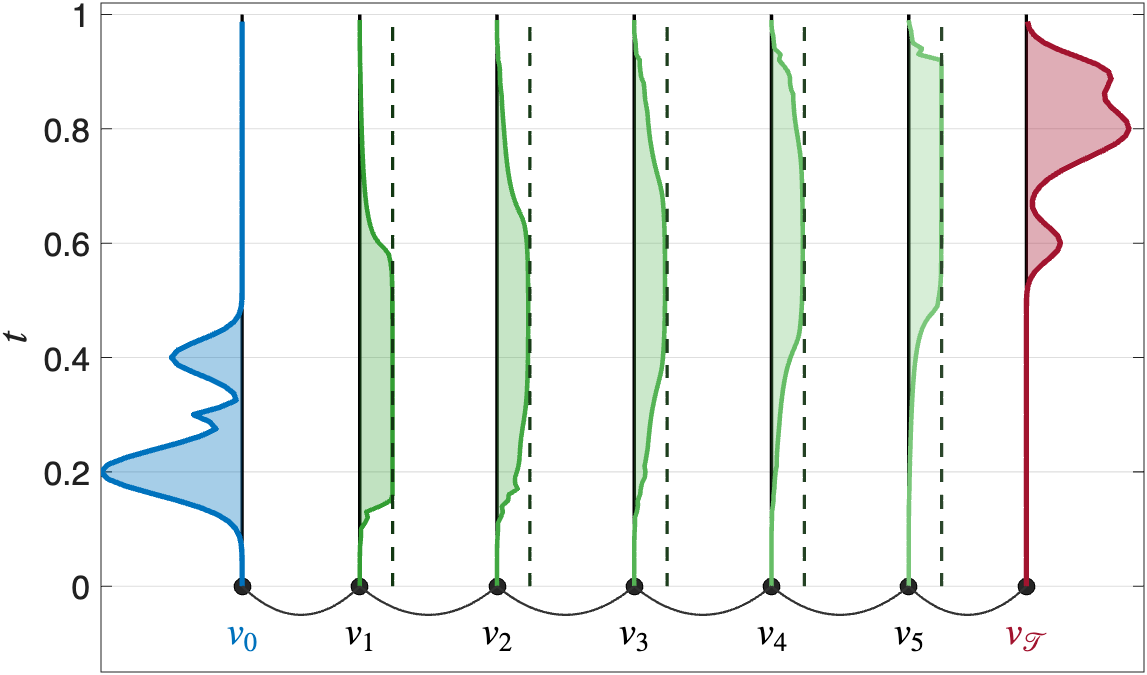}
\caption{\emph{Independent DA on a line graph.} Seven rails represent $v_0$ (source), five interior crossings $v_1$-$v_5$, and $v_{\mathcal T}$ (arrival). Blue is the source, green the interior crossing-time rails with per-stage capacity markers, and red the arrival. The optimizer reallocates mass across stages to satisfy capacity at each interior rail while preserving the source and arrival marginals.}
\label{fig:independent_DA_line}
\end{figure}

\subsection{Coupled DA constraints on line graph}\label{sec:coupled-num}
Now, we fix the DA coupling and optimize only the crossing-time schedule under a per-time-slice capacity. The structure mirrors the independent DA case at the crossing node, but the boundary coupling is no longer a variable. With pairs $(v_0,v_{\mathcal T})$ fixed, the schedule routes mass through time without changing who is matched to whom. The solution on the line is unique. Monotone order need not hold: capacity can delay some early departures and advance some late ones, so temporal order may flip even as feasibility, conservation, and the prescribed coupling are enforced. Fig.~\ref{fig:placeholder} shows these reversals (smoothness as in Remark~\ref{remark:smooth}).

\begin{figure}[htb!]
\centering
\includegraphics[width=0.65\linewidth]{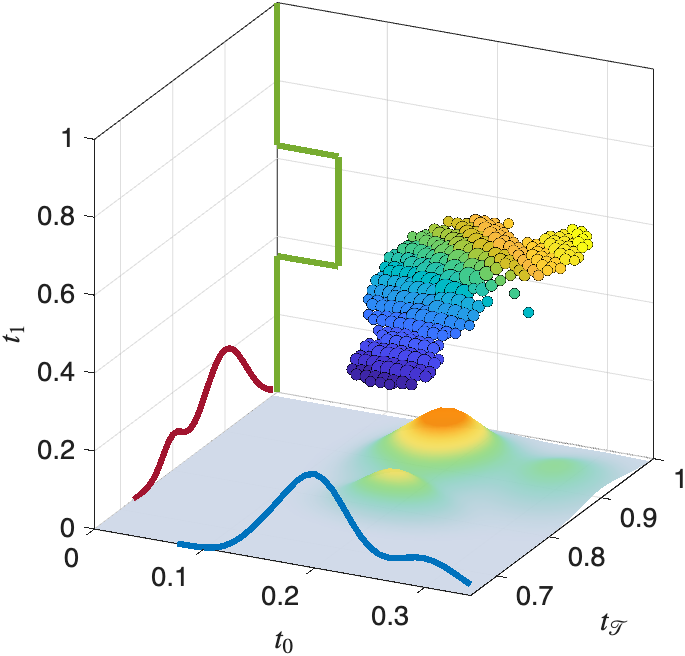}
\caption{\emph{Coupled DA with fixed DA coupling.} The surface over $(t_0,t_{\mathcal T})$ is the prescribed coupling between $v_0$ and $v_{\mathcal T}$. Blue and red ridges reproduce the boundary marginals. The green curve shows the crossing-time profile (extended by zero outside its window). Colored markers display dominant triplets $(t_0,t_1,t_{\mathcal T})$ and reveal order reversals: the unique schedule respects capacity but does not preserve monotone time order.}
    \label{fig:placeholder}
\end{figure}

\subsection{Transportation on path-specified networks}\label{sec:path-network-num}

Finally, we restrict the feasible set to three admissible paths that advance one level at a time and meet at shared nodes, with capacities adding within every common node. In the instance shown in Fig. \ref{fig:path_network}, the routes are $\bp_1 = \{v_0,v_2, v_3, v_4, v_6, v_{\mathcal T}\}$, $\bp_2 = \{v_0, v_1, v_3, v_4, v_5, v_{\mathcal T}\}$, and $\bp_3 = \{v_0, v_2, v_3, v_4, v_5, v_{\mathcal T}\}$. Flow is confined to these paths and may split or merge at shared nodes. Every interior node enforces a uniform per-time capacity $r_v(t)\le \bar c$ with $\bar c=1.4$, so the green ribbons are clipped at the common dashed cap while the boundary profiles at $v_0$ and $v_{\mathcal T}$ remain fixed.

\begin{figure}[htb!]
\centering
\includegraphics[width=0.75\linewidth]{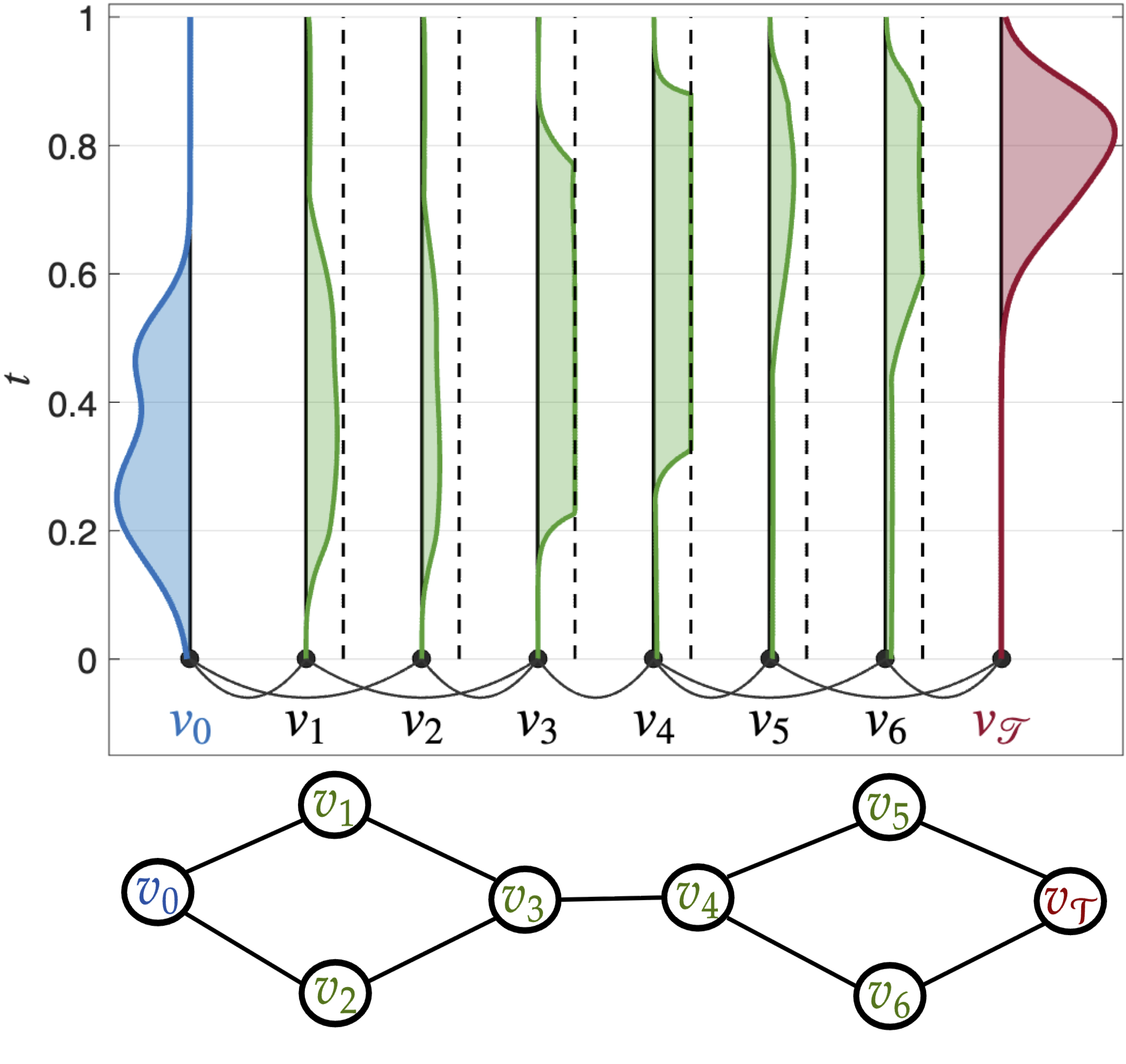}
\caption{\emph{Path-specified network.} Three admissible paths connect $v_0$ to $v_{\mathcal T}$ and intersect only at nodes. Blue and red rails display the departure and arrival time profiles. Green vertical ribbons at nodes on the paths $\bp_1$, $\bp_2$, and $\bp_3$, showing crossing-time profiles with support on $[0,1]$. The visualization highlights path-wise scheduling with splitting and merging at shared junctions under uniform node capacities.}
\label{fig:path_network}
\end{figure}

\black

\subsection{Convergence at linear rate}\label{sec:convergence-num}
The boundary temporal marginals $\mu_0$ and $\mu_{\mathcal T}$ are fixed Gaussian mixtures concentrated near the start and end of the horizon. Each interior node uses a single time-profile multiplier $w_v(t)$ shared across all paths. At each iteration, we propagate forward and backward path messages, aggregate $m_0$, $m_{\mathcal T}$, $m_v$, rescale $u$ and $v$ to match $\mu_0$ and $\mu_{\mathcal T}$, and clip $w_v$ to the capacity profile. Fig. \ref{fig:convergence} shows that the marginal errors $\|\hat{\mu}_0-\mu_0\|_1$ and $\|\hat{\mu}_{\mathcal T}-\mu_{\mathcal T}\|_1$ decay, indicating a linear convergence rate.

\begin{figure}[htb!]
\centering
\includegraphics[width=.65\linewidth]{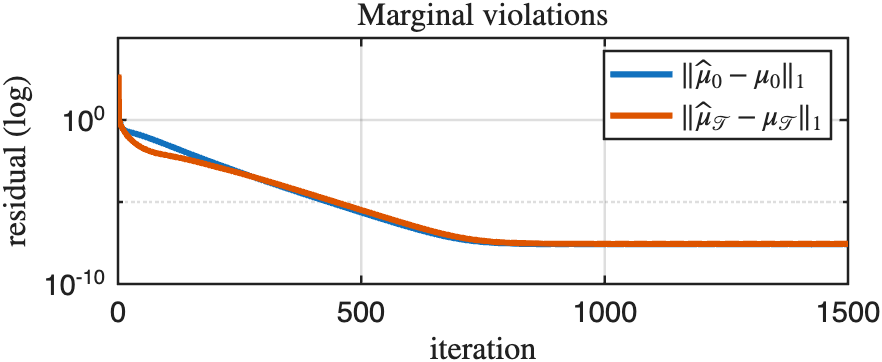}
\caption{Marginal violations for the path-wise Sinkhorn solver on the three-path grid. The plot shows the log-scale $\ell_1$ errors $\|\hat{\mu}_0-\mu_0\|_1$ and $\|\hat{\mu}_{\mathcal T}-\mu_{\mathcal T}\|_1$ over $1500$ iterations, where $\hat{\mu}_0$ and $\hat{\mu}_{\mathcal T}$ are the estimated departure and arrival marginals. The curves are approximately straight lines on the log scale, consistent with a linear convergence rate.}
\label{fig:convergence}
\end{figure}

\section{Conclusion}\label{sec:conclusion}

In this paper, we developed a comprehensive framework for OT over networks subject to DA constraints, extending classical formulations by incorporating temporal specifications and nodal flux limitations at sources and sinks. In particular, two types of DA constraints were considered: independent constraints, where departure and arrival rates are prescribed separately, and coupled constraints, where each particle’s transportation timespan is explicitly specified. We established feasibility, existence, and uniqueness in both settings. The independent DA formulation aligns naturally with a multi-marginal OT structure, while the coupled case gives rise to an unequal-dimensional OT problem, reflecting joint timing constraints. To address the arising optimization challenges from our problem, we began by analyzing transportation over line graphs, where we derived analytical conditions for feasibility and uniqueness based on monotonicity and the generalized Monge condition. We also introduced a node aggregation strategy for path reduction. We proposed efficient numerical solving for the resulting high-dimensional problem by considering its entropic regularization, and develop customized Sinkhorn-type algorithms for both the independent and coupled departure--arrival constricted setting. The scalable algorithm improved applicability for large-scale network optimization. Future work may include incorporating origin–destination (OD) structure together with DA profiles directly at the path level. It may also explore refined graph reductions where large networks are approximated by DAG skeletons that preserve path-wise timing, and extend the framework to richer physical constraints such as speed limits, buffer sizes, and fundamental-diagram-type flow bounds, aiming toward deployable scheduling tools for specified systems.

\section*{Acknowledgment}
This research has been supported in part by the Swedish Research Council Distinguished Professor Grant 2017-01078, Knut and Alice Wallenberg Foundation Wallenberg Scholar Grant, and the Swedish Research Council (VR) under grant 2020-03454, KTH Digital Futures.

\bibliographystyle{plain}  
\bibliography{references}

\end{document}